\documentclass{siamltex}
\usepackage[margin=1.5in]{geometry}
\usepackage{amsmath,amsfonts,amssymb,mathtools}
\usepackage{url}
\usepackage{latexsym}
\usepackage{graphicx,overpic}
\usepackage{caption}
\usepackage{pgfplots}
\pgfplotsset{compat=newest}
\usepackage{color}
\usepackage{booktabs}
\usepackage{array}
\newcolumntype{C}[1]{>{\centering\let\newline\\\arraybackslash\hspace{0pt}}m{#1}}
\usepackage[hidelinks]{hyperref}
\hypersetup{
    colorlinks=true,
    linkcolor={red!50!black},
    citecolor={blue!50!black},
    urlcolor={blue!80!black}
}

\newcommand{\rrangle}{\rangle\kern-1.2ex~\rangle}
\newcommand{\llangle}{\langle\kern-1.2ex~\langle}

\usepackage{tikz}

\newcommand{\R}{\mathbb R}
\newcommand{\bA}{\mathbf A}

\newcommand{\bC}{\mathbf C}

\newcommand{\bD}{\mathbf D}

\newcommand{\bH}{\mathbf H}
\newcommand{\bI}{\mathbf I}

\newcommand{\bP}{\mathbf P}

\newcommand{\bM}{\mathbf M}

\newcommand{\cV}{\mathcal{V}}
\newcommand{\bV}{\mathbf V}

\newcommand{\ba}{\mathbf a}
\newcommand{\bb}{\mathbf b}
\newcommand{\bg}{\mathbf g}
\newcommand{\bn}{\mathbf n}

\newcommand{\bT}{\mathbf T}
\newcommand{\bu}{\mathbf u}
\newcommand{\blf}{\mathbf f}
\newcommand{\bU}{\mathbf U}

\newcommand{\bv}{\mathbf v}
\newcommand{\bc}{\mathbf c}
\newcommand{\Gs}{\mathcal{S}}

\newcommand{\bhu}{\widehat{\mathbf u}}
\newcommand{\bw}{\mathbf w}

\newcommand{\bW}{\mathbf W}
\newcommand{\bx}{\mathbf x}
\newcommand{\by}{\mathbf y}
\newcommand{\bz}{\mathbf z}

\newcommand{\Vd}{V_{\hskip-0.0ex 1}}
\newcommand{\Hd}{V_{\hskip-0.0ex 0}}

\newcommand{\T}{\mathcal T}

\newcommand{\tr}{{\rm tr}}
\newcommand{\Div}{\mathop{\rm div}}
\newcommand{\divG}{{\mathop{\,\rm div}}_{\Gamma}}

\newcommand{\gradG}{\nabla_{\Gamma}}

\newcommand{\nablaG}{\nabla_{\Gamma}}

\newcommand{\cD}{\mathcal D}

\newcommand{\cT}{\mathcal T}

\newcommand{\cS}{\mathcal S}

\newcommand{\cW}{\mathcal W}

\newcommand{\bpsi}{\boldsymbol{\psi}}
\newcommand{\tbpsi}{\widetilde{\boldsymbol{\psi}}}
\newcommand{\bxi}{\boldsymbol{\xi}}

\newcommand{\vect}[1]{\boldsymbol{\mathbf{#1}}}

\newcommand{\pc}{\partial^{\circ}}
\newcommand{\Gn}{{\Gamma_{\hskip-0.25ex n}}}
\newcommand{\curlG}{{\mathop{\,\rm curl}}_{\Gamma}}

\newcommand*\diff{\mathop{}\!\mathrm{d}}

\usepackage{subcaption}
\usepackage{graphicx}
\newcommand{\includegraphicsw}[2][1.]{\includegraphics[width=#1\linewidth]{#2}}
\usepackage{multirow}
\usepackage{hhline}
\usepackage{float} 

\newenvironment{xtabular}[2][1]{\def\arraystretch{#1}\tabular{#2}}{\endtabular}

\DeclareMathOperator{\Curl}{curl}

\usepackage{accents}
\renewcommand*{\dot}[1]{%
	\accentset{\mbox{\large\bfseries .}}{#1}}

\newtheorem{remark}{Remark}[section]

\numberwithin{equation}{section}
\begin{document}

\title{Tangential Navier-Stokes equations on evolving surfaces: Analysis and simulations\thanks{This
work was  partially supported by US National Science Foundation (NSF) through grants DMS-1953535 and DMS-2011444.}}
\author{
Maxim A. Olshanskii\thanks{Department of Mathematics, University of Houston, Houston, Texas 77204 (molshan@math.uh.edu)}
\and
Arnold Reusken\thanks{Institut f\"ur Geometrie und Praktische Mathematik, RWTH-Aachen University, D-52056 Aachen,
	Germany (reusken@igpm.rwth-aachen.de)}
\and  Alexander Zhiliakov\thanks{Department of Mathematics, University of Houston, Houston, Texas 77204 (alex@math.uh.edu)}
}
\maketitle

\begin{abstract}  
The paper considers a system of equations that models a lateral flow of a Boussinesq--Scriven fluid  on a passively evolving   surface embedded in $\mathbb{R}^3$.  
For the resulting Navier-Stokes type system, posed on a smooth closed time-dependent surface, we introduce a weak formulation in terms of functional spaces
on a space-time manifold defined by the surface evolution. 
The weak formulation is shown to be well-posed  for any finite final time and without smallness conditions on data. 
We further extend an unfitted finite element  method, known as  TraceFEM, to compute solutions to the fluid system. 
Convergence of the method is demonstrated numerically. In another series of experiments we visualize lateral flows induced by smooth deformations of a material surface.  
\end{abstract}

\section{Introduction} \label{sectIntroduction}
There is extensive literature on analysis and numerical simulation of the incompressible Navier-Stokes equations,  a fundamental model of fluid mechanics. While the overwhelming majority of  papers in this field treats these equations  in Euclidean domains, there also is literature on analysis of the incompressible Navier-Stokes equations on surfaces, or more general on Riemannian manifolds. 
Building on a fundamental observation made by Arnold~\cite{arnold1966geometrie} that relates equations of incompressible fluid to  finding geodesics on the group of all volume preserving diffeomorphisms,  local existence and uniqueness results for  Navier-Stokes equations on compact oriented Riemannian  manifolds were proved in the seminal paper \cite{ebin1970groups}. 
This work has been followed by many other studies, cf.  \cite{Temam88,taylor1992analysis} and the overview in \cite{Chan2017}. Very recent activity in the field includes the work \cite{Pruss2021,simonett2021h}, in which local-in-time-well-posedness in the framework of maximal regularity is established. All these papers restrict to \emph{stationary} surfaces or manifolds.  

In recent years there has been a growing interest in fluid equations on \emph{evolving} surfaces \cite{Yavarietal_JNS_2016,Gigaetal,Jankuhn1,miura2017singular,reuther2018solving}, motivated in particular by applications to modeling of biological membranes, e.g.~\cite{saffman1975brownian,rangamani2013interaction,barrett2015numerical,torres2019modelling}. In \cite{brandner2021derivations} one finds an overview and comparison of different modeling approaches for evolving  viscous fluid layers that result in the surface Navier-Stokes equations. We are not aware of any literature presenting well-posedness analysis of this system on evolving  surfaces. Furthermore, only very few papers address numerical treatment of such equations. 
In \cite{reuther2015interplay,reuther2018aErratum} computational results are presented, based on a surface vorticity-stream function formulation of the Navier-Stokes equations. The surface motion is prescribed and the evolving SFEM of Dziuk-Elliott~\cite{Dziuk07,DEreview} is applied to the partial differential equations for the scalar vorticity and stream function unknowns. The authors of  \cite{Nitschkeetal_arXiv_2018} consider another discretization approach  that is based on the techniques developed in \cite{reuther2018solving}. These papers focus on modeling and illustration of certain interesting flow phenomena but not on the performance of the numerical methods.  
Several recent papers~\cite{jankuhn2020error,bonito2020divergence,hansbo2016analysis,Jankuhn2020} present error analysis of finite element discretization methods for vector valued PDEs on \emph{stationary} surfaces.
We are not aware of any paper with a systematic numerical study or an error analysis of a discretization method for {vector} valued PDEs on {evolving} surfaces. 
We conclude that in the field of incompressible Navier-Stokes equations on \emph{time-dependent} surfaces basic problems related to well-posedness of the systems, development and analysis of numerical methods remain open.  This paper addresses two of these problems: well-posedness and numerical method development. 

It is shown in  \cite{brandner2021derivations} that several different modeling approaches all yield the same \emph{tangential} surface Navier-Stokes equations (TSNSE). These equations govern the evolution of tangential velocity and surface pressure if  the normal velocity of the surface is {prescribed}. The main topic of the present paper is the analysis of a variational formulation of the TSNSE. In particular, a  well-posedness result for this formulation is proved. To the best of our knowledge,  this is the first well-posedness result for evolving surface Navier-Stokes equations. The paper also touches on the development of a new discretization method for the TSNSE. This method combines an implicit time stepping scheme with a TraceFEM~\cite{ORG09,olshanskii2016trace} for discretization in space. We explain this method, validate its optimal second order convergence for a test problem with a known solution and apply it to the simulation of a lateral flow induced by deformations of a sphere.  Error analysis of this method is not addressed in this paper and left for future research.

The remainder of the paper is organized as follows. In section~\ref{sectioncont} we recall the surface Navier-Stokes equations known from the literature. In particular, the \emph{tangential} surface Navier-Stokes equations are described.  
Appropriate function spaces for a variational formulation of the TSNSE are introduced in section~\ref{sectPrelim}. Relevant properties of these spaces are derived. The main results of this paper are given in section~\ref{sec4}. We introduce and analyze two variational formulations of the TSNE: The first one is for the tangential velocity only, which is solenoidal by construction of the solution space. Then we introduce the pressure and study a mixed variational problem. For both formulations well-posedness results are derived. In section~\ref{s:discretization} we explain a discretization method. Finally, section~\ref{s_num} collects and discusses results of numerical experiments.

\section{Surface Navier--Stokes equations} \label{sectioncont}
We first introduce necessary notations of surface quantities and tangential differential operators. 
For a closed smooth  surface $\Gamma$ embedded in $\mathbb{R}^3$,
the outward pointing normal vector  is denoted by $\bn$, and  the normal projector on the tangential space at $\bx\in\Gamma$ is $\bP=\bP(\bx)=\bI- \bn\bn^T$. 
Let $\bH=\nabla_\Gamma\bn \in\mathbb{R}^{3\times3}$ be the  Weingarten mapping (second fundamental form) and $\kappa:=\tr(\bH)$ twice the mean curvature.
For  a scalar function $p:\, \Gamma \to \mathbb{R}$ or a vector field $\bu:\, \Gamma \to \mathbb{R}^3$   their smooth extensions to a neighborhood $\mathcal{O}(\Gamma)$  of $\Gamma$ are denoted by $p^e$ and $\bu^e$, respectively. Surface gradients and covariant derivatives on $\Gamma$ can be defined through derivatives in $\mathbb{R}^3$   as $\nablaG p=\bP\nabla p^e$, $D_\Gamma \bu\coloneqq \nabla \bu^e \bP$, and  $\nabla_\Gamma \bu\coloneqq \bP \nabla \bu^e \bP$. These definitions  are  independent of a particular smooth extension of $p$ and $\bu$ off $\Gamma$.
The surface rate-of-strain tensor \cite{GurtinMurdoch75} is given by
 $E_s(\bu)\coloneqq\frac12(\nabla_\Gamma \bu + \nabla_\Gamma \bu^T)$,
the surface divergence and $\Curl_\Gamma$ operators for a vector field $\bu: \Gamma \to \R^3$ are  $\divG \bu  \coloneqq \tr (\gradG \bu)$ and $\Curl_\Gamma \bu  \coloneqq (\gradG\times \bu)\cdot\bn$. For a tensor field $\bA=[\ba_1,\ba_2,\ba_3]: \Gamma \to \mathbb{R}^{3\times 3}$ ,  $\divG \bA$  is defined row-wise and $D_\Gamma\bA$ is a third-order tensor such that $\big(D_\Gamma\bA\big)_{i,j,k}= \big(D_\Gamma\ba_j\big)_{i,k}$. 

We now let $\Gamma(t)$ be a \textit{material} surface embedded in $\mathbb{R}^3$ as defined in \cite{GurtinMurdoch75,MurdochCohen79}, with a density distribution $\rho(t,\bx)$.  By $\bu(t,\bx)$, $\bx \in \Gamma(t)$,  we denote a velocity field  of the density flow on $\Gamma$, i.e. $\bu(t,\bx)$ is the velocity of a material point  $\bx\in\Gamma(t)$.  The derivative $\dot f$ of a surface quantity $f$ along the corresponding trajectories of material points is called the material derivative.
Assuming the surface evolution is such that the space-time manifold 
\[
\Gs=\bigcup\limits_{t\in[0,T]}\{t\} \times \Gamma(t)\subset\mathbb{R}^4
\] is smooth, the material derivative can be defined as
\begin{equation} \label{matder}
\dot f = \frac{\partial f^e}{\partial t}+ (\bu\cdot\nabla) f^e\quad\text{on}~\Gs,
\end{equation}
where $f^e$ is a smooth extension  of  $f:\,\Gs\to\mathbb{R}$  into a spatial neighborhood of $\Gs$.
Note that $\dot f$ is a tangential derivative for $\Gs$, and hence it depends only on the surface values of $f$ on $\Gs$. For a vector field $\bv$ on $\Gs$, one defines $\dot\bv$ componentwise.

The conservation of mass and linear momentum for a thin material layer represented by $\Gamma(t)$ together with the Boussinesq--Scriven constitutive relation for the surface stress tensor and an inextensibility condition  leads to the \emph{surface Navier-Stokes equations}: 
\begin{equation} \label{momentum}
	\left\{
	\begin{split}
		\rho \dot \bu & =- \nabla_\Gamma \pi + 2\mu \divG (E_s(\bu))  + \bb +  \pi \kappa\bn \\
		\divG \bu  & =0\\
		\dot{\rho} & =0
	\end{split}\right.\qquad\text{on}~\Gamma(t),
\end{equation}
where $\pi$ is the surface fluid pressure and $\mu$ stands for the viscosity coefficient.
Equations \eqref{momentum} model the evolution of an inextensible viscous fluidic material surface with acting area force $\bb$, cf. \cite{Gigaetal,Jankuhn1} for  derivations of this model and  \cite{brandner2021derivations} for a literature overview and alternative forms of this system.
The pure geometrical evolution of $\Gamma(t)$  is defined by its normal velocity $V_\Gamma=V_\Gamma(t)$ that is given by
the normal component of the material velocity,
\begin{equation} \label{Gamma_evol}
	V_\Gamma=\bu\cdot \bn \quad\text{on}~\Gamma(t).
\end{equation}
If $\bb$ is given or defined through other unknowns, then \eqref{momentum}--\eqref{Gamma_evol} form a closed system of six equations for six unknowns $\bu$, $\pi$, $\rho$, and $V_\Gamma$, subject to suitable initial conditions.

\subsection{Tangential surface Navier-Stokes equations} 
We now introduce a major simplification by assuming that the \emph{geometric} evolution of $\Gamma$ is known. We make this more precise below and derive equations governing the unknown lateral motions of the surface fluid.
To this end, consider a smooth velocity field $\bw=\bw(t,\bx)$ in $[0,T]\times\Bbb{R}^3$ that passively advects the embedded surface $\Gamma(t)$  given by
\begin{equation} \label{defevolving}
	\Gamma(t)=\{\by\in\R^3~|~\by=\bx(t,\bz),~\bz \in \Gamma_0 \},
\end{equation} 
where the trajectories $\bx(t,\bz)$ are the unique solutions of the Cauchy problem
\begin{align}
	\begin{cases}\label{def_flow_map}
		\bx(0,\bz)=\bz\\
		\frac{d}{dt} \bx(t,\bz) = \bw(t, \bx(t,\bz)),
	\end{cases}
\end{align}
for all $\bz$ on an initial smooth connected surface  $\Gamma_0=\Gamma(0)$ embedded in $\mathbb{R}^3$.
We now assume that the normal material motion of $\Gamma$ is completely determined by the ambient flow $\bw$ and the lateral material motion is free, i.e., for the \emph{given} $\bw$ the relation
\begin{equation}\label{wu}
\bu\cdot\bn=\bw\cdot\bn\quad\text{on}~\Gamma(t)
\end{equation}
holds for the normal component\footnote{For velocity fields $\bv \in \Bbb{R}^3$ defined on $\Gamma(t)$  we use a splitting  into tangential and normal components $\bv=\bv_T+\bv_N=\bv_T+v_N\bn$, with $v_N=\bv\cdot \bn$.} $u_N=\bu \cdot \bn$, while the tangential component $\bu_T$ of the surface fluid flow is unknown and depends on $\bw$ only implicitly through the variation of $\Gamma(t)$ and conservation laws represented by equations \eqref{momentum}. The resulting system can be seen as an idealized model for the motion of a fluid layer embedded in  bulk fluid, where one neglects friction forces between the surface and bulk fluids as well as any effect of the layer on the bulk flow. In such a physical setting, \eqref{wu} means non-penetration of the bulk fluid through the material layer.

Material trajectories of points on the surface are defined by the flow field $\bu$, rather than $\bw$.  We are also interested in a derivative determined by the variation of a quantity along the so-called normal trajectories defined below. 
\begin{definition}\label{defN} \rm
Let  $\Phi_t^n:\, \Gamma_0 \to \Gamma(t)$, be the flow map of the pure geometric (normal) evolution of the surface, i.e.,  for $\bz \in \Gamma_0$, the \emph{normal trajectory} $\bx^n(t,\bz)=\Phi_t^n(\bz)$  solves 
\begin{align}
	\begin{cases}\label{def_flow_map_normal}
		\bx^n(0,\bz)=\bz,\\
		\frac{d}{dt} \bx^n(t,\bz)= \bw_N(t,\bx^n(\bz,t)).
	\end{cases}
\end{align}
Eq. \eqref{def_flow_map_normal} defines a bijection between $\Gamma_0$ and $\Gamma(t)$ for every $t\in[0,T]$ with inverse mapping  $\Phi_{-t}^n$.
The Lagrangian derivative for the flow map $\Phi_t^n$ is denoted by $\pc$:
 \begin{equation} \label{Lagr}
  \pc \bv(t,\bx) =  \frac{d}{dt} \bv (t,\Phi_t^n(\bz)), \quad \bx=\Phi_t^n(\bz).
 \end{equation}
We call $\pc \bv$ the \emph{normal time derivative} of $\bv$.
\end{definition}

It is clear from \eqref{Lagr} that this normal time derivative is an intrinsic surface quantity. Similar to the material derivative in \eqref{matder}, it can be expressed in terms of bulk derivatives if one assumes a smooth extension of   $\bv$ from $\Gs$ to its neighborhood:  
\begin{equation}\label{eq:aux286}
\pc \bv(t,\bx)  = \frac{d}{dt} \bv (t, \bx^n(t,\bz)) 
=\Big(\frac{\partial \bv^e}{\partial t}+ (\bw_N\cdot\nabla) \bv^e \Big)(t,\bx)
\end{equation}
for $(t,\bx)\in\Gs$.
Comparing the material and normal time derivatives of a flow field $\bv$ on the surface we find the equality  
\[
  \bP \dot \bv 
         = \bP \pc \bv + (\nablaG \bv) \bu_T.
\]
 With the splitting $\bv=\bv_T+\bv_N$ we get
\begin{equation} \label{eq3}
 \bP \dot \bv= \bP \pc \bv_T+\bP \pc \bv_N + (\nablaG \bv_T) \bu_T+ (\nablaG \bv_N) \bu_T. 
\end{equation}
Noting that $\bP \bn=0$ and $\bP \pc \bn= - \nablaG w_N$ (cf. (2.16) in \cite{Jankuhn1}), we rewrite  $\bP \pc \bv_N$  as
\[
 \bP \pc \bv_N = \pc v_N \bP \bn + v_N \bP \pc \bn = v_N \bP \pc \bn= - v_N \nablaG w_N.
\]
We also have the relation $(\nablaG \bv_N) \bu_T = v_N \bH \bu_T$. Using these results and letting $\bv=\bu$ in \eqref{eq3} one obtains
\begin{equation} \label{eq4}
 \bP \dot \bu= \bP \pc \bu_T+ (\nablaG \bu_T) \bu_T+w_N \bH \bu_T -\tfrac12 \nablaG w_N^2,
\end{equation}
where we also used $u_N=w_N$. To derive an equation for the unknown tangential velocity $\bu_T$, we  apply the projection $\bP$ to the  first equation in \eqref{momentum}. For $\bP \dot \bu$ we have the result \eqref{eq4}. Note that the term   $\tfrac12 \nablaG w_N^2$ is known and can be treated as a source term. 
For a stationary surface ($w_N=0$) the normal time derivative is just the usual time derivative, $\bP \pc \bu_T = \frac{\partial \bu_T}{\partial t}$. The term $(\nablaG \bu_T) \bu_T$ is the analog of the quadratic term in Navier-Stokes equations. 
Using $\divG \bu_N= u_N \kappa$ and $u_N=w_N$, the second equation in \eqref{momentum} yields  $\divG \bu_T=-w_N \kappa$. We are not interested in variable density case and let $\rho=1$. Summarizing, from the surface Navier-Stokes equations \eqref{momentum} we get the following reduced system for $\bu_T$ and $\pi$  which we call the \emph{tangential} surface Navier-Stokes equations (TSNSE):
\begin{equation}\label{NSalt}
	\left\{\begin{aligned}
		\bP \pc\bu_T +  (\nablaG \bu_T) \bu_T+ w_N \bH \bu_T- 2\mu  \bP \divG E_s(\bu_T) +\nabla_\Gamma\pi&=\blf \\
		\divG \bu_T   &= g,\\
	\end{aligned}\right.
\end{equation} 
with right-hand sides known in terms of geometric quantities,  $w_N$ and the tangential component of the external area force $\bb$:
\begin{equation}\label{rhs}
 g = - w_N \kappa,\qquad \blf=\bb_T+ 2\mu  \bP \divG(w_N\bH) + \tfrac12\gradG w_N^2.
\end{equation} 
In the remainder of this paper we study this TSNSE. Note that these equations have a structure similar to the standard incompressible Navier-Stokes equations in Euclidean domains. Important  differences are that TSNSE is formulated on a space-time manifold that does not have an evident tensor product structure and, related to this,  a normal time derivative $\bP \pc$ instead of the standard time derivative is used and an additional term $w_N \bH \bu_T$ occurs.   
After some preliminary results in the next section,  we introduce a well-posed weak formulation of the TSNSE in section~\ref{sec4}. 
\begin{remark} \rm If one does \emph{not} assume a given normal velocity $u_N=w_N$, an equation for $u_N$ can be derived from \eqref{momentum}, cf. \cite{Jankuhn1}. The surface Navier-Stokes equations \eqref{momentum} are then rewritten as a \emph{coupled system} for $\bu_T$, $\pi$ and $u_N$, that consists of TSNSE \eqref{NSalt} and the coupled equation
\begin{equation} \label{NSnormal}
  \dot u_N  = -2\mu \big(\tr(\bH \gradG \bu_T)+u_N \tr(\bH^2)\big) + \bu_T\cdot \bH \bu_T -\bu_T\cdot \gradG u_N +\pi \kappa + b_N.
\end{equation}
 A challenging problem, not addressed in this paper, is the well-posedness of the surface Navier-Stokes equations \eqref{momentum}, i.e., of the coupled system \eqref{NSalt}--\eqref{NSnormal}. For studying this problem, results on well-posedness of only the TSNSE \eqref{NSalt} may be useful.    
\end{remark}

\section{Preliminaries} \label{sectPrelim}
In this section we introduce several function spaces and derive relevant properties of these spaces.
We will use these spaces  to formulate a well-posed weak  formulation of the TSNSE \eqref{NSalt}.  At this point, we make our assumptions on $\Gamma_0$ and its evolution more precise. We introduce the following smoothness assumptions:
\begin{equation}\label{Assump}
	\Gamma_0\in C^3\quad\text{and}\quad \bw\in C^3([0,T]\times\mathbb{R}^3,\mathbb{R}^3),~\sup_{[0,T]\times\mathbb{R}^3}|\bw|<+\infty.
\end{equation}
Then  the ODE system~\eqref{def_flow_map}  has a unique solution for any  $\bz\in\Gamma_0\subset\mathbb{R}^3$, which defines a one-to-one mapping $\Gamma_0\to\Gamma(t)$ for all $t\in[0,T]$ (Theorems~II.1.1,~V.3.1 and remark to Theorem~V.2.1 in \cite{Hartman}). Moreover, this mapping is $C^3(\cS_0,\mathbb{R}^4)$ (Corollary~V.4.1~\cite{Hartman}) with
\[\cS_0:=[0,T] \times \Gamma_0.\] 
Therefore, $\cS$ is a $C^3$ manifold as the image of $\cS_0\in C^3$ under a smooth mapping. 

We need a globally $C^2$-smooth extension of the spatial normal $\bn(t,x)$, $(t,x) \in \Gs$ that can be constructed as follows.
Let $\phi_0$ be the signed distance function to $\Gamma_0$. On a tubular neighborhood $U_\delta$ of $\Gamma_0$, with diameter $\delta > 0$ sufficiently small, we have $\phi_0 \in C^3(U_\delta)$, cf. \cite[Lemma 2.8]{DEreview}. We extend this function to be from $C^3(\mathbb{R}^3)$ and zero outside $U_{2\delta}$. Thus we have $\phi_0 \in C^3(\R^3)$ and $\phi_0$ is a signed distance function in a neighborhood of $\Gamma_0$. Let $\Phi_t$ be the flow map for the velocity field $\bw$. The mapping $(t,\bx) \to \Phi_t(\bx)$ is $C^3([0,T]\times \R^3, \R^3)$ and $\nabla_x  \Phi_t(\bx)$ is regular \cite{Hartman}. Define the level set function $\phi(t,x):=\phi_0\big(\Phi_{-t}(\bx)\big)$ and the neighborhood $\Gs^{\rm ex}:= \cup_{t \in [0,T]} \{t\} \times \Phi_t(U_\delta)$ of $\Gs$. Then we have $\phi \in C^3([0,T]\times \R^3, \R)$ and for $(t,\bx) \in \Gs^{\rm ex}$ it holds $|\nabla \phi(t,x)| \geq c >0$, and $\phi(t,x)=0$ iff $(t,x) \in \Gs$.  Set  $\hat \bn(t,\bx):= \nabla \phi(t,x)/|\nabla \phi(t,x)|$ for $(t,\bx) \in \Gs^{\rm ex}$. Clearly $\hat \bn =\bn$ on $\Gs$ and $\hat \bn \in C^2(\Gs^{\rm ex}, \R^3)$, and by a standard procedure we can extend it to $\hat \bn \in C^2([0,T] \times \R^3, \R^3)$. To simplify the notation, this extension is denoted by $\bn$.
For such an extended  vector field $\bn$ we have that $\bw_N=(\bw\cdot\bn)\bn \in  C^2([0,T]\times\mathbb{R}^3, \R^3)$ holds. Arguing in the same way as above, we conclude that for the normal flow mapping from definition~\ref{defN}
we have
\begin{equation}\label{ass:smooth}
	\Phi_{(\cdot)}^n\in C^2(\cS_0,\cS).
\end{equation}  
Note that $\cS_0 = \overline{\cS_0}$ and $\cS = \overline{\cS}$, i.e.,  $\cS_0$ and $\cS$ are closed manifolds.
\\[1ex]
 %
We need function spaces suitable for a weak formulation of the TSNSE. For this we make use of a general framework of evolving spaces presented in \cite{alphonse}.  In section~\ref{secspace1} we introduce specific evolving Hilbert spaces, based on a Piola pushforward mapping. Based on results from \cite{alphonse} several properties of these spaces are derived. In section~\ref{secspace2} an evolving space of functions for which suitable weak ``material'' derivatives exist is introduced. Here we deviate from \cite{alphonse} in the sense that this ``material'' derivative  is not based on the pushforward map but on the normal time derivative defined above.     

\subsection{Surface Piola transform} \label{secPiola}
To define evolving Hilbert spaces based on standard Bochner spaces, we need a suitable pushforward map. In the context of this paper it is natural to use a surface Piola transform as pushforward map, since this transform  conserves the solenoidal property of a tangential vector field.

To define a surface variant of the Piola transform based on the {normal} flow map $\Phi_t^n :\, \Gamma_0 \to \Gamma(t)$, we need some further notation. Below we always take $\bz \in \Gamma_0$ and $\bx:=\Phi^{n}_t(\bz) \in \Gamma(t)$. 
Since for each $t\in[0,T]$ the map $\Phi_t^n:\, \Gamma_0 \to \Gamma(t)$ is a $C^2$-diffeomorphism,  the differential  $D\Phi_t^n(\bz): \, (T\Gamma_0)_{\bz} \to T\Gamma(t)_\bx$,  is invertible.  Define  $J=J(t,\bz):=\det D\Phi_t^n(\bz)$, $J^{-1}=J^{-1}(t,\bx)= \det D\Phi_{-t}^n(\bx)= J(t,\bz)^{-1}$.
Denote by $\bD=\bD(t,\bz)$ and $\bD ^{-1}=\bD^{-1}(t,\bx)$  the matrices of linear mappings given by $D\Phi_t^n(\bz)\bP(\bz):\mathbb{R}^3\to\mathbb{R}^3$ and $D\Phi_{-t}^n(\bx) \bP(\bx)=[D\Phi_t^n(\bz)]^{-1}\bP(\bx):\mathbb{R}^3\to\mathbb{R}^3$, respectively. Note that $\bD^{-1} \bD = \bP(\bz)$ and $\bD \bD^{-1} = \bP(\bx)$ hold. For these mappings the following useful identities hold: 
\begin{equation}\label{aux551}
	\begin{split}
		D_{\Gamma}\,(\bv\circ\Phi_{-t}^n)=  \left(D_{\Gamma_0}\bv\right)\bD^{-1}\quad\text{for}~\bv\in C^1(\Gamma_0)^3,\\
		\left(D_{\Gamma}\,\bv\right)\bD=  D_{\Gamma_0}(\bv\circ\Phi_{t}^n)
		\quad \text{for}~\bv\in C^1\big(\Gamma(t)\big)^3.
	\end{split}
\end{equation}
We need the Piola transform for arbitrary, not necessarily tangential vectors. 
For this it is convenient to define an invertible operator ${\bA}(t,\bz): \Bbb{R}^3\to\Bbb{R}^3$ 
such that $\bA|_{T\Gamma_0}=J^{-1}D\Phi_t^n:\, T\Gamma_0\to T\Gamma(t)$ and $\bA:T\Gamma_0^\perp\to T\Gamma(t)^\perp$. We use the operator
\begin{equation}\label{Adef}
	{\bA}(t,\bz)\bv  :=J^{-1}(t,\bx)\bD(t,\bz)\bv+\bn_{\Gamma(t)}(\bx)\bn_{\Gamma_0}(\bz)\cdot \bv\,, \quad   \bv  \in \Bbb{R}^3.
\end{equation}
For  $	{\bA}^{-1}(t,\bx) :=J(t,\bz)\bD^{-1}(t,\bx)+ \bn_{\Gamma_0}(\bz) \bn_{\Gamma(t)}(\bx)^T$
it holds ${\bA}^{-1}(t,\bx){\bA}(t,\bz)=I_{\Bbb{R}^3}$.
The  matrices of  $\bA$ and $\bA^{-1}$ in the standard basis are also denoted by $\bA$ and $\bA^{-1}$, respectively. Note that  $\det \bA=1$ holds.  
We define the  \emph{surface Piola transform} $P_t: \Bbb{R}^3\to\Bbb{R}^3$ by
\begin{equation} \label{Piola}
	(P_t\bv)(\bx):= \bA(t,\bz)\,\bv(\bz), \quad \bz \in \Gamma_0.
\end{equation}
This operator maps tangential vectors on $\Gamma_0$ to tangential vectors on $\Gamma(t)$ and for \emph{tangential} vectors $\bv$ it satisfies $\divG P_t\bv = 0$ a.e. on $\Gamma(t)$ iff $\divG \bv =0$ a.e. on $\Gamma_0$, cf. \cite{steinmann2008boundary}. 

We need some regularity properties of $\bD$, $\bA$ and $\bD^{-1}$, $\bA^{-1}$, which are collected in the following lemma. For a function $g \in C^1({\cS}_0)$ the maximum norm is $\|g\|_{C^1(\cS_0)}:= \max_{(t,\bz) \in \cS_0}\big(|g(t,\bz)| +|\nabla_{S_0}g(t,\bz)|\big)$ and similarly for vector and matrix valued functions as well as for such functions on $\cS$. 

\begin{lemma}\label{lemS} It holds that $\bD, \bA \in C^1({\cS}_0)^{3\times3}$,  $\bD^{-1},\bA^{-1}\in C^1({\cS})^{3\times3}$ and, in particular, 
	\begin{equation}\label{smooth_est}
		\|J\|_{C^1({\cS}_0)}+\|\bD\|_{C^1({\cS}_0)}+\|\bA\|_{C^1({\cS}_0)}+\|J^{-1}\|_{C^1({\cS})}+\|\bD^{-1}\|_{C^1({\cS})}+\|\bA^{-1}\|_{C^1({\cS})}\leq C.
	\end{equation}
\end{lemma}
\begin{proof} From \eqref{ass:smooth} we know that $\Phi: (t,\bz) \to 	(t, \Phi_{t}^n(\bz))$ is in $ C^2({\cS}_0,{\cS})$ and hence $D\Phi\in C^1(T{\cS}_0,T{\cS})$. Moreover, $\bP_{\cS_0}$ is $C^1$-smooth, so  $D\Phi\bP_{\cS_0}$ is a $C^1$ smooth mapping with  matrix representation
	\begin{equation} \label{A1}
		\left[\begin{matrix}
			1 & \bw_N^T \\
			0 & \bD
		\end{matrix}\right] \in C^1({\cS}_0)^{4\times4}.
	\end{equation} 
	Hence,  $\bD\in C^1({\cS}_0)^{3\times3}$ 
	and $J\in C^1({\cS}_0)$ hold. Combining this with $\bn_{\Gamma_0},
	\bn_{\Gamma(\cdot)}\circ\Phi_{(\cdot)}^n\in C^1({\cS}_0)$,  \eqref{Adef} and the property that $\cS_0$  is closed, implies  the bound in  \eqref{smooth_est}	for $\bD$, $J$ and $\bA$.
	The mapping $\Phi:\, \cS_0 \to \cS$ is one-to-one. By the inverse mapping theorem the inverse $\Phi^{-1}$  
	is $ C^2(\Gs,{\cS}_0)$ and  for its differential  we have   $D\Phi^{-1}\in C^1(T{\cS},T{\cS}_0)$. The matrix of the $C^1$ smooth mapping $D\Phi^{-1}\bP_{\cS}$ is 
	\[
	\left[\begin{matrix}
		1 & - \bw_N^T  \bD^{-1} \\
		0 & \bD^{-1}
	\end{matrix}\right]  \in C^1({\cS})^{4\times4}.
	\]
	This and  ${\bA}^{-1}  =J\bD^{-1}	+\bn_{\Gamma_0} \bn_{\Gamma(\cdot)}^T(\Phi_{(\cdot)}^n)$ imply the desired bound for  $\bD^{-1}$,  $J^{-1}$ and $\bA^{-1}$.
\end{proof}
\smallskip

\subsection{Evolving Hilbert spaces} \label{secspace1}
 For constructing suitable evolving Hilbert spaces, we first define \emph{tangential} velocity spaces on $\Gamma(t)$.
The notation  $(\cdot,\cdot)_{0,t}$ and $\|\cdot\|_{0,t}$  is used for the canonical inner product and norm in $L^2(\Gamma(t))$.   We need the Sobolev spaces of order one  
  \[
      H^1(t):=\{\bv\in H^1(\Gamma(t))^3~|~\bv\cdot\bn=0~~\text{on}~~\Gamma(t)\},
  \]
with the inner product $(\cdot,\cdot)_{1,t}:=(\cdot,\cdot)_{0,t}+(D_\Gamma\cdot,D_\Gamma\cdot)_{0,t}$,  
 and its closed subspace of divergence free tangential fields,  
 \[
  	\Vd(t) :=\{\, \bv \in H^1(t)~|~\divG \bv=0 ~~\text{a.e. on}~~\Gamma(t)\,\}.
  \]
The space  $\Hd(t)$ is defined as  closure of a space of  smooth div-free tangential fields in the $L^2(\Gamma(t))$ norm: 
\[    
 \Hd(t):= \overline{\cV(t)}^{\|\cdot\|_{0,t}},\quad	\cV(t) := \{\, \bv \in C^1(\Gamma(t))^3~|~ \bv \cdot \bn=0, ~\divG \bv=0~~\text{on}~~\Gamma(t)\,\}. \\	
\]
The space of smooth functions $\cV(t)$ is  dense not only in $\Hd(t)$ but also in $\Vd(t)$. Indeed, for any tangential velocity field $\bu\in L^2(\Gamma(t))^3$ on the $C^2$ smooth surface $\Gamma(t)$ 
we have a  Helmholz decomposition $\bu=\nabla_\Gamma \psi + \bn\times(\nabla_\Gamma \phi) + \mathbf{h}$  with some $\psi,\phi\in H^1(\Gamma(t))$  and a harmonic field $\mathbf{h}\in C^1(\Gamma(t))^3$~\cite{reusken2020stream}. For
$\bu\in \Vd(t)$ we have $\psi=0$ and the result follows from the density of $C^2$-smooth functions in $H^1(\Gamma(t))$.
Therefore, endowed with canonical scalar products, the spaces form a  Gelfand triple $\Vd(t) \hookrightarrow \Hd(t) \hookrightarrow \Vd(t)'$. We also have that the dense embedding $\Vd(t) \hookrightarrow \Hd(t)$ is compact.  Here and later $H'$ always denotes a dual of a Hilbert space $H$, and we adopt the common notation $H^{-1}(t)$ for  $H^1(t)'$.

For the space $L^2(t):=\{\bv\in L^2(\Gamma(t))^3~|~\bv\cdot\bn=0~\text{a.e. on}~\Gamma(t)\}$  we define  a \emph{pushforward map  $\phi_t:L^2(0)\to L^2(t)$, based on  the Piola transform}, by 
\begin{equation} \label{Pio2} (\phi_{t}\bv)(\bx)=(P_t \bv)(\bx)= \bA(t,\bz) \bv(\bz), \quad \bv\in L^2(0),~ \bx=\Phi_t^n(\bz),\, \bz \in \Gamma_0. 
\end{equation}
The inverse  map (pullback) is given by  $(\phi_{-t}\bv)(\bz)= \bA^{-1}(t,\bx) \bv(\bx)$, $\bv \in L^2(t)$. 
Since  $\bA\in C^1({\cS}_0)^{3\times 3}$, the  restriction of $\phi_{t}$  to $H^1(0)$ is a pushforward map 
from $H^1(0)$ to $H^1(t)$. Because $\phi_t$ is based on the Piola transform and thus conserves the solenoidal property, we also have that  $\phi_t$  is a pushforward map from $\Hd(0)$ to $\Hd(t)$,  and from $\Vd(0)$ to $\Vd(t)$. 
For this pushforward map we have for $\bv\in H^1(0)$:
\[
\begin{split}
\|\phi_{t}\bv\|_{1,t}&= \left(\|\bA \,\bv\circ\Phi_{-t}^n\|^2_{0,t}
+ \|D_\Gamma(\bA \,\bv\circ\Phi_{-t}^n)\|^2_{0,t}
\right)^{\frac12}\\
&\le \left(\|\bA\|_{C(\Gamma(t))}+
\|D_\Gamma \bA\|_{C(\Gamma(t))}\right) \|\bv\circ\Phi_{-t}^n\|_{0,t}
+
\|\bA\|_{C(\Gamma(t))} \|D_\Gamma (\bv\circ\Phi_{-t}^n)\|_{0,t}\\
&\le \left(\|\bA\|_{C(\Gamma(t))}+
\|D_\Gamma \bA\|_{C(\Gamma(t))}\right)\|J\|_{C(\Gamma_0)}^{\frac12} \|\bv\|_{0,0}\\
&\qquad +
\|\bA\|_{C(\Gamma(t))} \|\bD^{-1}\|_{C(\Gamma(t))} \|J\|_{C(\Gamma_0)}^{\frac12} \|D_\Gamma\bv\|_{0,0}.
\end{split}
\]
The result \eqref{smooth_est} implies that the norms $\|\bA\|_{C(\Gamma(t))}$,  $\|\bD^{-1}\|_{C(\Gamma(t))}$, 
$\|D_\Gamma \bA\|_{C(\Gamma(t))}$, 
$\|J\|_{C(\Gamma_0)}$ are uniformly bounded in $t$ and thus
\[
\sup_{t\in[0,T]}\|\phi_{t}\bv\|_{1,t} \le C \|\bv\|_{1,0}
\]
holds with some $C$ independent of $\bv\in H^1(0)$. 
With  similar arguments one easily shows that $\|\phi_{-t}\bv\|_{1,0} \le C \|\bv\|_{1,t}$ holds for all $\bv\in H^1(t)$, with  $C$ independent of $\bv$ and $t$.  
These bounds  remain obviously true if $H^1(0)$, $H^1(t)$ and the corresponding norms are replaced by $\Hd(0)$, $\Hd(t)$ and the corresponding norms. Using  \eqref{smooth_est} one shows that the maps $t\to\|\phi_{t}\bv\|_{1,t}$ and $t\to\|\phi_{t}\bv\|_{0,t}$ are continuous.
These properties imply that  $\{\Hd(t),\phi_t\}_{t\in [0,T]}$, $\{H^1(t),\phi_t\}_{t\in [0,T]}$, and  $\{\Vd(t),\phi_t\}_{t\in [0,T]}$ are ``\emph{compatible pairs}'' in the sense of Definition 2.4 in \cite{alphonse}.
This compatibility structure induces natural properties of the evolving spaces defined as follows:
%
\[ \begin{split}
 L_{\Vd}^2 & := \{\, \bv:\, [0,T] \to \bigcup_{t\in [0,T]}\{t\} \times \Vd(t),\, t \to (t,\bar \bv (t))~|~\phi_{-(\cdot)}\bar \bv(\cdot) \in L^2(0,T; \Vd(0))\,\}, \\
 L_{\Vd'}^2 & := \{\, \bg:\, [0,T] \to \bigcup_{t\in [0,T]}\{t\} \times \Vd(t)',\, t \to (t,\bar \bg (t))~|~\phi_{(\cdot)}^\ast \bar \bg(\cdot) \in L^2(0,T; \Vd(0)')\,\},\\
 L^2_{\Hd} &:= \{\, \bv:\, [0,T] \to\bigcup_{t\in [0,T]} \{t\} \times \Hd(t),\, t \to (t,\bar \bv (t))~|~\phi_{-(\cdot)}\bar \bv(\cdot) \in L^2(0,T; \Hd(0))\,\},
\end{split} \]
where $\phi_{t}^\ast$ is the dual of $\phi_t$. We shall also need the spaces  $L_{\Hd}^\infty$, $L^2_{H^1}$, $L^2_{H^{-1}}$, which are defined analogously, and the spaces of smooth space-time functions 
\begin{equation} \label{defD} \begin{split}
	\cD &= \{\, \bv \in L_{\Vd}^2~|~\phi_{-(\cdot)}\bar \bv(\cdot) \in C^\infty\big([0,T]; \cV(0))\,\},
	\\
	\cD_0 &= \{\, \bv \in L_{\Vd}^2~|~\phi_{-(\cdot)} \bv(\cdot) \in C_0^\infty \big((0,T); \cV(0))\,\}.
\end{split}
\end{equation}
Note that functions in $\cD_0$ have zero traces on $\partial \Gs$.
 With a slight abuse of notation we identify $\bar \bv(t)$ with $\bv(t)= (t,\bar \bv(t))$.
 
In \cite{alphonse} it is shown  that  if $\Vd(0) \hookrightarrow \Hd(0) \hookrightarrow \Vd(0)'$ is a Gelfand triple, with a compact embedding  $\Vd(0) \hookrightarrow \Hd(0)$,  and both $\{\Hd(t),\phi_t\}_{t\in [0,T]}$ and  $\{\Vd(t),\phi_t\}_{t\in [0,T]}$ are compatible pairs, then the     $L$-spaces inherit certain  properties of the standard Bochner spaces. In particular, cf. Section 2 in \cite{alphonse}, the spaces $L_{\Vd}^2$ and $L^2_{\Hd}$
with
\[
  (\bu,\bv)_{\bf 1}  = \int_0^T (\bu(t),\bv(t))_{1,t} \, dt, \quad
  (\bu,\bv)_{\bf 0} = \int_0^T (\bu(t),\bv(t))_{0,t} \, dt,
\]
are separable Hilbert spaces, homeomorphic to $L^2(0,T;\Vd(0))$ and $L^2(0,T;\Hd(0))$ respectively. Furthermore, the embedding $ L_{\Vd}^2 \hookrightarrow L^2_{\Hd}$ is dense and compact,  the  dual space $(L_{\Vd}^2)'$  is isometrically isomorphic to $L_{\Vd'}^2$, and 
\[ L_{\Vd}^2 \hookrightarrow L^2_{\Hd}\hookrightarrow L_{\Vd'}^2
\] 
is a Gelfand triple.   
The space $\cD_0$ is dense in $L_{\Vd}^2$ and so $\cD_0$ is dense also in $L^2_{\Hd}$. By the same arguments $L_{H^1}^2$ is also a Hilbert space with inner product $(\cdot,\cdot)_{\bf 1}$. The subspace of smooth functions 
\begin{equation} \label{D0hat}
\widehat{\cD}_0 = \{\, \bv \in L_{H^1}^2~|~\phi_{-(\cdot)} \bv(\cdot) \in C_0^\infty \big((0,T); C^{1}(\Gamma_0,T\Gamma_0)\big)\,\}
\end{equation} 
is dense in $L_{H^1}^2$ and $\left(L_{H^1}^2\right)'\simeq L_{H^{-1}}^2$ holds. Note that $L_{\Vd}^2$ is a closed subspace of $L_{H^1}^2$ and that functions in ${\cD}_0$ are solenoidal and functions in $\widehat{\cD}_0$ are not necessarily solenoidal.

\subsection{Some uniform inequalities}
We need to establish several basic inequalities on $\Gamma(t)$ with constants \emph{uniformly} bounded in $t$. 

We first consider a Korn inequality. Recall that the estimate
\begin{equation} \label{Korn}
	\|\bv\|_{1,t} \leq c \Big( \|\bv\|_{0,t} + \|E_s(\bv)\|_{0,t}\Big) \quad \text{for all}~\bv \in H^1(t)
\end{equation}
holds with a constant $c=c(t)$ that depends on smoothness properties of $\Gamma(t)$, cf. \cite{Jankuhn1}.
In the next lemma we  show that the constant can be taken such that $\max_{t \in [0,T]} c(t) < \infty$ holds.
\newtheorem{Proposition}{Proposition}
\begin{lemma} \label{Propkorn} The constant $c$ in \eqref{Korn} can be chosen finite and independent of $t$.
\end{lemma}
\begin{proof}
	Fix any $t\in[0,T]$ and $\bv \in H^1(t)$.  Define $\bu=\bD ^{T}\bv\circ\Phi^{n}_t \in H^1(0)$. Below, for the $3$-tensor $\mathbb{T}=D_\Gamma\bD ^{T}$ and $\bv\in\mathbb{R}^3$,  $\mathbb{T}\bv\in\mathbb{R}^{3\times3}$ is the 2nd-mode tensor--vector product.  With the help of \eqref{aux551} one computes
	\begin{equation*}
		\begin{split}
			\nabla_{\Gamma_0}\bu(\bz) &= \bP(\bz)[D_{\Gamma_0}\bu(\bz)] =
			\bP(\bz)[D_{\Gamma_0}(\bD ^{T}\bv\circ\Phi^{n}_t(\bz))]\\
			&=
			\bP(\bz)\mathbb{T}\bv(\bx)
			+\bP(\bz)\bD ^{T}D_\Gamma\bv(\bx)\bD ,\\
			&=
			\bP(\bz)\mathbb{T}\bv(\bx)
			+\bD ^{T}\nabla_\Gamma\bv(\bx)\bD ,
			\quad \bx=\Phi^{n}_t(\bz).  
		\end{split}
	\end{equation*}
	The latter inequality holds since $\bD ^{T}\bP(\bx)=\bP(\bz)\bD ^{T}$. From this we find
	\begin{equation} \label{aux653}
		\begin{split}
			E_s(\bu)(\bz) &= \frac{1}{2}\left(\bP(\bz) \mathbb{T}\bv(\bx)
			+ \big(\mathbb{T}\bv(\bx)\big)^T\bP(\bz) 
			\right) 
			+ \bD ^{T}E_s(\bv)(\bx)\bD .  
		\end{split}
	\end{equation}
	With the help of $D_\Gamma \bv = D_\Gamma (\bD^{-T} \bu \circ \Phi_{-t}^n)$, \eqref{smooth_est}, \eqref{Korn} applied for $t=0$, and \eqref{aux653} we estimate 
	\[
	\begin{split}
		\|D_\Gamma\bv&\|_{0,t} = \|J^{\frac12}\left([D_\Gamma\bD ^{-T}]\bu +\bD^{-T}(D_\Gamma\bu)\bD ^{-1}\right)\|_{0,0}\\
		& \le C(\|\bu\|_{0,0}+\|D_\Gamma\bu\|_{0,0}) \le C\left(\|\bu\|_{0,0} +\|E_s(\bu)\|_{0,0}\right)\\ 
		& = C\left(\|J^{-\frac12}\bD ^{T}\bv\|_{0,t} +\left\|J^{-\frac12}\left(\tfrac{1}{2}\big(
		\bP \mathbb{T}\bv
		+ \big(\mathbb{T}\bv\big)^T\bP
		\big) 
		+ \bD ^{T}E_s(\bv)\bD \right)\right\|_{0,t}\right)\\ 
		& \le C\left(\|J^{-\frac12}\bD ^{T}\bv\|_{0,t} +\|J^{-\frac12}\bP\mathbb{T}\bv\|_{0,t} +  \|J^{-\frac12}\bD ^{T}E_s(\bv)\bD \|_{0,t}\right)\\ 
		& \le C\left(\|\bv\|_{0,t}+ \|E_s(\bv)\|_{0,t}\right).	
	\end{split}
	\]
	with some $C$ independent of $t\in[0,T]$ and $\bv$.\quad
\end{proof}
\smallskip

The following inf-sup estimate holds~\cite{Jankuhn1}:
\begin{equation} \label{infsup}
	\|\nabla_\Gamma\pi\|_{H^{-1}(t)}:= \sup_{0\neq\bv\in H^1(t)}\frac{\int_{\Gamma(t)}\pi\divG\bv\,ds}{\|\bv\|_{1,t}}  \geq c(t)\|\pi\|_{0,t},\quad \forall\,\pi\in L^2(\Gamma(t)),~\int_{\Gamma(t)}\pi=0, 
\end{equation}
with $c(t)>0$.  A uniformity result for this inf-sup constant is derived in the following lemma. 
\begin{lemma} \label{leminfsup}
	The constant $c(t)$ in \eqref{infsup} can be taken such that $\inf_{t \in [0,T]} c(t) > 0$ holds.
\end{lemma}
\begin{proof}
	We use a similar approach as in the proof of the previous lemma, and derive an estimate on $\Gamma(t)$ by pulling forward the result on $\Gamma_0$. We  use the pullforward  $\phi_t$ that is based on the Piola transform and satisfies $\Div_{\Gamma(t)} (\phi_t \bw)(\bx)= J^{-1} \Div_{\Gamma_0} \bw (\bz)$, $\bz \in \Gamma_0$, $\bx =\Phi_t^n(\bz) \in \Gamma(t)$.  Take $\bv\in H^1(t)$ and $\pi\in L^2(\Gamma(t))$ with $\int_{\Gamma(t)}\pi=0$. Define $c:=-|\Gamma_0|^{-1} \int_{\Gamma_0} \pi \circ \Phi_t^n \, ds$ and  $\bw:=\phi_{-t}\bv \in H^1(0)$. Note that $\|\bv\|_{1,t} \le C\|\bw\|_{1,0}$ with a constant $C$ uniformly bounded in $t\in [0,T]$  (compatibility property). We have
	\[
	(\pi, {\Div}_{\Gamma(t)} \bv)_{0,t}=\big(\pi, {\Div}_{\Gamma(t)}(\phi_t \bw)\big)_{0,t} =(\pi\circ \Phi_t^n, {\Div}_{\Gamma_0} \bw)_{0,0}. 
	\]
	Using this and the result \eqref{infsup} for $t=0$ we get:
	\begin{align*}  &C\sup_{0\neq\bv\in H^1(t)}\frac{\int_{\Gamma(t)}\pi\divG\bv\,ds}{\|\bv\|_{1,t}}  \ge \sup_{0\neq\bw\in H^1(0)}\frac{\int_{\Gamma_0} (\pi\circ \Phi_t^n) {\Div}_{\Gamma_0} \bw\, ds}{\|\bw\|_{1,0}} \\ & \geq c(0)\| \pi\circ \Phi_t^n +c\|_{0,0} \geq c(0) \|J^\frac12\|_{C(\cS)}^{-1}\| \pi +c\|_{0,t}  \\ & =c(0) \|J^\frac12\|_{C(\cS)}^{-1}\big(\| \pi\|_{0,t} + c |\Gamma(t)|^\frac12\big) \geq c(0) \|J^\frac12\|_{C(\cS)}^{-1}\| \pi\|_{0,t},
	\end{align*}
	which yields a $t$-independent strictly positive lower bound for $c(t)$ in \eqref{infsup}.
\end{proof}
\smallskip

We now derive a uniform interpolation estimate.
\begin{lemma} The interpolation inequality (Ladyzhenskaya's inequality)
	\begin{equation} \label{Lady}
		\|\bv\|_{L^4(\Gamma(t))} \leq C \|\bv\|_{0,t}^\frac12 \|\bv\|_{1,t}^\frac12, \quad \bv\in H^1(t).
	\end{equation}
	holds with a constant $C < \infty$  independent of $t$.
\end{lemma}
\begin{proof}   
	Consider  $v\in H^1(\Gamma(t))$ and let $\hat v = v\circ\Phi_{-t}^n$. For a compact Riemann manifold $\Gamma_0$, the estimate (II.38) from \cite{Aubin} yields  $\|\hat v\|_{L^4(\Gamma_0)} \leq C \|\hat v\|_{0,0}^\frac12 \|\hat v\|_{1,0}^\frac12$. An examination of the proof shows that the estimate remains true if $\Gamma_0$ is a $C^2$ compact manifold. With the help of this estimate applied component-wise and \eqref{smooth_est} we calculate for $\bv\in H^1(t)$
	\[
	\begin{split}
		\|\bv&\|_{L^4(\Gamma(t))}= \|\hat \bv J^\frac{1}{4}\|_{L^4(\Gamma_0)}\le   C \|J\|^{\frac14}_{C(\Gs_0)} \|\hat \bv\|_{L^4(\Gamma_0)}
		\leq C \|\hat \bv\|_{0,0}^\frac12 \|\hat \bv\|_{1,0}^\frac12 \\
		& \leq C \|J^{-\frac{1}{2}} \bv\|_{0,t}^\frac12\left( \|J^{-\frac{1}{2}} \bv\|_{0,t}+ \|J^{-\frac{1}{2}} \bD^{-T}\nablaG \bv\|_{0,t} \right)^{\frac12}\\
		& \leq C \|\bv\|_{0,t}^\frac12\left( \|\bv\|_{0,t}+  \|\nablaG \bv\|_{0,t} \right)^{\frac12}
		\leq C\,\|\bv\|_{0,t}^\frac12 \| \bv\|_{1,t}^\frac12,
	\end{split}
	\] 
	with some $C$ independent of $t$.\quad
\end{proof}
\smallskip

For  $\bxi\in H^1(t)$ consider the Helmholtz decomposition (see e.g. \cite{reusken2020stream})
\begin{equation} \label{Helmholtz}
\bxi =  \bxi_1+ \bxi_2,\quad \text{with}~\bxi_1=\nabla_\Gamma \phi,~\phi\in H^1(\Gamma(t))~\text{and}   ~\bxi_2\in V_1(t).
\end{equation}
\begin{lemma} \label{LemmaHelm} For $\bxi_i$ as in \eqref{Helmholtz} we have $\bxi_i\in H^1(t)$ and $\|\bxi_i\|_{1,t}\le C\|\bxi\|_{1,t}$, $i=1,2$,
	 with a constant $C$ finite and independent of $t$.
\end{lemma}
\begin{proof}
 Due to the $L^2$ orthogonality of the Helmholtz decomposition we have $\|\bxi_1\|_{0,t}^2 +\|\bxi_2\|_{0,t}^2=\|\bxi\|_{0,t}^2$. Also note that $\divG \bxi_2=0$, $\divG \bxi=\divG \bxi_1$, $\curlG \bxi_1=0$, $\curlG \bxi=\curlG \bxi_2$. Furthermore on $H^1(t)$ we have the norm equivalence $\|\bu\|_{1,t} \sim \|\bu\|_{0,t} + \|\divG \bu \|_{0,t} +\|\curlG \bu\|_{0,t}$. A $t$-dependence in the  constants in this norm equivalence enters only through the Gaussian curvature of $\Gamma(t)$, cf. \cite[Theorem~3.2]{reusken2020stream}. Due to the smoothness property $\cS \in C^3$ the Gaussian curvature is uniformly bounded on  $\cS$ and thus the constants in this norm equivalence can be taken independent of $t$.
Using these results we get
\[
\|\bxi_1\|_{1,t} \le C(\|\bxi_1\|_{0,t} + \|
\divG \bxi_1 \|_{0,t} +\|\curlG \bxi_1\|_{0,t})=C(\|\bxi_1\|_{0,t} + \|
\divG \bxi \|_{0,t})  \le C \|\bxi\|_{1,t},
\]
and by similar arguments $\|\bxi_2\|_{1,t}\le C\|\bxi\|_{1,t}$
 with a constant $C$ uniformly bounded in $t$.
\end{proof}

\subsection{Solution space} \label{secspace2}
In this section we  introduce a subspace of $L_{\Vd}^2$ consisting of  functions for which a suitable weak normal time derivative exists. This space will be the solution space in the weak formulation of TSNSE. 

We recall the Leibniz rule 
\[
  \frac{d}{dt} \int_{\Gamma(t)}  v\, ds = \int_{\Gamma(t)} (\pc v + v \divG \bw_N)\, ds = 
    \int_{\Gamma(t)} (\pc v + v w_N \kappa)\, ds,
\]
Thus for velocity fields $\bv,\bu \in C^1(\Gs)$ we get
\begin{equation} \label{PI1}
  \frac{d}{dt} \int_{\Gamma(t)} \bv \cdot \bu \, ds=\int_{\Gamma(t)}\left( \pc (\bv \cdot \bu) +   (\bv \cdot \bu)  w_N \kappa \right)\, ds. 
\end{equation}
This implies the integration by parts identity
\begin{equation} \label{PIid} \begin{split}
 & \int_0^T\int_{\Gamma(t)}( \pc \bv \cdot \bu +\bv \cdot \pc \bu  + (\bv \cdot \bu)  w_N \kappa)\, ds\,dt \\   & = \int_{\Gamma(T)} \bv \cdot \bu \, ds -\int_{\Gamma_0} \bv \cdot \bu \, ds, \quad \bv,\bu \in C^1(\Gs)^3. 
\end{split}\end{equation}
Based on this we define for $\bv \in L_{H^1}^2$ the normal time derivative as the functional $\pc \bv$:
\begin{equation} \label{weak}
 \langle \pc \bv, \bxi\rangle:= - \int_0^T\int_{\Gamma(t)}\left(  \bv \cdot \pc \bxi  + (\bv \cdot \bxi)  w_N \kappa\right)\, ds\,dt, \quad \bxi \in \widehat{\cD}_0.
\end{equation}
Note that  functions in $\widehat{\cD}_0$ are not necessarily solenoidal, cf. \eqref{D0hat}.
Restricting now to $ L_{\Vd}^2\subset  L_{H^1}^2$, assume $\bv \in L_{\Vd}^2$ is such that
\[
  \|\pc \bv\|_{(L_{\Vd}^2)'}:= \sup_{\bxi \in \cD_0}\frac{\langle \pc \bv, \bxi\rangle}{
  \|\bxi\|_{\bf 1}}
\]
is bounded.  Since $\cD_0$ is dense in $L_{\Vd}^2$, $\pc \bv$ can 
 then be extended to a bounded linear functional on $L_{\Vd}^2$. We use $(L_{\Vd}^2)'\cong L_{\Vd'}^2$ and introduce the space
\[ \begin{split} 
  \bW(\Vd,\Vd') &= \{\, \bv \in L_{\Vd}^2~|~\pc \bv \in L_{\Vd'}^2\,\}, \quad \text{with}\\
  (\bv,\bu)_W &:=\int_0^T (\bv(t),\bu(t))_{1,t} +(\pc \bv(t),\pc \bu(t))_{\Vd(t)'}\, dt.
  \end{split}
\]
This space is used as solution space in the weak formulation of TSNSE below. In the remainder of this section we derive certain useful properties of this space. For this it will be helpful to introduce in addition to the Lagrangian derivatives $\dot \bv$ (material derivative) and $\pc \bv$ (normal time derivative) one other Lagrangian derivative, which is  based on the  pushforward operator $\phi_t$:
\begin{equation}\label{eq:Past}
	\partial^\ast \bv(t):= \phi_t\left(\frac{d}{dt} \phi_{-t}\bv(t)\right), \quad \bv\in \cD.
\end{equation}
The reason that we introduce the $\partial^\ast$ derivative is, that it is the same as the one used in the general framework in \cite{alphonse} and we can use results derived in that paper.
Note that the $	\partial^\ast$ derivative is defined for tangential flow fields and based on the Piola transform implying 
\begin{equation} \label{aux709}
\bn\cdot\partial^\ast \bv=0\quad\text{and}\quad \divG \partial^\ast \bv=0~~ \text{for}~\bv\in \cD.
\end{equation}
We now derive   relations between the derivatives $\partial^\ast$ and $\pc $. 

\begin{lemma} \label{lemrelation}
	For $ \bv\in \cD$ 
	the following holds:
	\begin{align}
		\pc \bv & = \partial^\ast \bv - {\bA} (\pc {\bA}^{-1}) \bv,\label{eq:ident1}\\
		\bP	\pc \bv & = \partial^\ast \bv - \bA \bP (\pc {\bA}^{-1}) \bv.\label{eq:ident2}
	\end{align}
\end{lemma}
\begin{proof}
Using the definitions of the pushforward and pullback mappings we compute 
	\[
	\begin{aligned}
	  \phi_t\left(\frac{d}{dt} \phi_{-t}\bv(t)\right)(\bx) &= \bA(t,\bz) \frac{d}{dt}\left[{\bA}^{-1}(t,\Phi_t^n(\bz))\bv(t,\Phi_t^n(\bz))\right] \\ 
		&= 
		\bA(t,\bz)\left(\pc \bA^{-1}(t,\bx)\bv(t,\bx) + {\bA}^{-1}(t,\bx) \pc \bv(t,\bx)\right)\\
		&= 
		\bA(t,\bz) \pc \bA^{-1}(t,\bx)\bv(t,\bx) + \pc \bv(t,\bx),
	\end{aligned}
	\]
	which yields the result \eqref{eq:ident1}.
	 The result \eqref{eq:ident2} follows from \eqref{eq:ident1} using $\bP\partial^\ast \bv=\partial^\ast \bv$ and $\bP \bA=\bP(\bx)\bA(t,\bz)= \bA(t,\bz)\bP(\bz)$.
\end{proof}
\smallskip

 From \eqref{eq:ident1} we obtain the identity
\[
\left(\partial^\ast \bv, \bxi\right)_{\bf 0}= 	\left(\pc \bv, \bxi\right)_{\bf 0} + \left(\bC \bv , \bxi\right)_{\bf 0} , \quad \forall\,\bv\in\cD,~\bxi \in \widehat{\cD}_0,
\] 
with $\bC:={\bA}\bP (\pc {\bA}^{-1})$.  Based on this, we define $\partial^\ast \bv$ for $\bv\in L_{\Vd}^2$ as the functional
\begin{equation}\label{aux755}
\langle \partial^\ast  \bv, \bxi\rangle: = \langle \pc  \bv, \bxi\rangle + \left(\bC \bv , \bxi\right)_{\bf 0}, \quad \bxi \in \widehat{\cD}_0.
\end{equation}
with $ \langle \pc  \bv, \bxi\rangle$ defined in \eqref{weak}.  
The density of $\widehat{\cD}_0$ in $L^2_{H^1}$ and  of ${\cD}_0\subset\widehat{\cD}_0$ in $L^2_{\Vd}$ allows  us to define $\partial^\ast \bv$ as an element of $ L^2_{H^{-1}}$ and $L^2_{\Vd'}$, respectively.
The following result holds:
\begin{equation}\label{VHeqv}
	\partial^\ast \bv\in L^2_{\Vd'}\quad\Longleftrightarrow\quad \partial^\ast \bv\in L^2_{H^{-1}}, \quad \bv\in L^2_{\Vd}.
\end{equation}
Implication ``$\Leftarrow$'' in  \eqref{VHeqv} is trivial since $\Vd\subset H^{1}$. To see ``$\Rightarrow$'', consider any 
$\bv\in \cD$  and $\bxi \in L^2_{H^1}$ together with its Helmholtz decomposition $\bxi=\nablaG\phi+\bxi_2$, cf. \eqref{Helmholtz}.
Thanks to Lemma~\ref{LemmaHelm} we get $\nablaG\phi \in L^2_{H^1}$, $\bxi_2 \in L_{\Vd}^2$ and $\|\nablaG\phi\|_{\bf 1}+ \|\bxi_2\|_{\bf 1} \le C\|\bxi\|_{\bf 1}$. Since $\bxi_2 \in L_{\Vd}^2$ we have
\begin{equation} \label{g1}
	|\langle \partial^\ast \bv, \bxi_2\rangle | \leq \|\partial^\ast\bv\|_{L^2_{\Vd'}} \|\bxi_2\|_{\bf 1} \le C \|\partial^\ast\bv\|_{L^2_{\Vd'}} \|\bxi\|_{\bf 1},
\end{equation}
while for the other component we get employing \eqref{aux709}
\begin{equation}\label{g2} 
	\langle \partial^\ast \bv, \nablaG\phi\rangle = ( \partial^\ast \bv, \nablaG\phi)_{\bf 0} =-(\divG \partial^\ast \bv, \phi)_{\bf 0} =0.
\end{equation}
We thus conclude $	|\langle \partial^\ast \bv, \bxi \rangle | \leq  C \|\partial^\ast\bv\|_{L^2_{\Vd'}} \|\bxi\|_{\bf 1}$ for all $\bv\in\cD$ and $\bxi\in L^2_{H^1}$. The result in \eqref{VHeqv} follows from the density of $\cD$ in $L^2_{\Vd}$.

\smallskip
We are now ready to prove the following result.
\begin{lemma}\label{Lem:W} The space $\bW(\Vd,\Vd')$ is a Hilbert space and $\cD$ is dense in $\bW(\Vd,\Vd')$. 
For any  $\bv\in\bW(\Vd,\Vd')$ and $t\in[0,T]$,  $\bv(t)$ is well-defined as an element of $\Hd(t)$ and it holds
\[
\sup_{t\in[0,T]} \|\bv(t)\|_{0,t}\le C \|\bv\|_W.
\]
\end{lemma}
\begin{proof} The idea of the proof is to relate the space $\bW(\Vd,\Vd')$ to the space $
 \bW_\ast(\Vd,\Vd') := \{\, \bv \in L_{\Vd}^2~|~\partial^\ast  \bv \in L_{\Vd'}^2\,\}$, with 
 $\|\cdot\|_{W_\ast}=(\|\cdot\|_{\bf 1}^2+\|\partial^\ast\cdot\|_{L^2_{\Vd'}}^2)^{\frac12}$,  and to show that the latter is homeomorphic to a standard Bochner space for $\cS_0$.
 Lemma~\ref{lemS} ensures $\bC\in C(\Gs)^{3\times3}$ and thus from \eqref{aux755} we obtain 
 \[
 |	\langle \partial^\ast \bv, \bxi\rangle|-c\,\|\bv\|_{\bf 0}\|\bxi\|_{\bf 0}
 \le |	\langle \pc \bv, \bxi\rangle|\le |	\langle \partial^\ast \bv, \bxi\rangle|+c\,\|\bv\|_{\bf 0}\|\bxi\|_{\bf 0}.
 \]
Therefore,  $\partial^\ast \bv$ is a linear bounded  functional on $L_{\Vd}^2$ iff $\pc \bv$ has this property. We conclude $\bv\in \bW(\Vd,\Vd')$ iff   $\bv\in \bW_\ast(\Vd,\Vd')$. Moreover, the above inequalities, definition of 
the $\bW(\Vd,\Vd')$-norm, $\bW_\ast(\Vd,\Vd')$-norm and $L_{\Vd}^2\hookrightarrow L^2_{\Hd}$ yield 
\[
c\|\bv\|_W \le \|\bv\|_{W_\ast} \le C\|\bv\|_W,
\]
with constants $0<c$ and $C<+\infty$ independent of $\bv\in\bW(\Vd,\Vd')$ and so  $\bW(\Vd,\Vd')=\bW_\ast(\Vd,\Vd')$ algebraically and topologically. 
Thus, it is sufficient to check the claims of the lemma for $\bW_\ast(\Vd,\Vd')$. For the latter we apply results 
from \cite{alphonse}, more specifically,
 Corollary 2.32 ($\bW_\ast(\Vd,\Vd')$ is a Hilbert space), Lemma 2.35 (continuous embedding $\bW_\ast(\Vd,\Vd') \hookrightarrow C([0,T]; \Hd(0))$) and Lemma 2.38 (density of smooth functions). For these results to hold one has to verify Assumption 2.31~\cite{alphonse}, which requires the mapping $\bv\to\phi_{-(.)}\bv$ to be a homeomorphism  between $\bW_\ast(\Vd,\Vd')$ and $\cW(\Vd(0),\Vd(0)')$, the  
 standard Bochner space 
\[
\cW(\Vd(0),\Vd(0)') = \{\, \bv \in L^2((0,T),\Vd(0))~|~ \partial_t\bv \in L^2((0,T),\Vd(0)')\,\}.
\]
It remains to check this  homeomorphism property.
We already derived  the norm equivalence $\|\bv\|_{\bf 1}\simeq\|\phi_{-(.)}\bv\|_{L^2([0,T],\Vd(0))}$, cf. Section~\ref{secspace1}.  To relate the norms $\|\partial_t \phi_{-(\cdot)}\bv\|_{L^2((0,T),\Vd(0)')}$ and $\|\partial^\ast \bv\|_{L_{\Vd'}^2}$  we consider  the following equalities for $\bv\in L_{\Vd}^2$, $\bxi\in \cD_0$, 
$\widetilde\bxi=\phi_{-(\cdot)}\bxi\in C_0^\infty((0,T),\cV(0) )$ and $\bA^{-1}(t,\cdot):\,  T\Gamma(t)\to T\Gamma_0$: 
\begin{equation}\label{aux516}
\begin{split}
\langle \partial_t (\phi_{-(\cdot)}\bv), \widetilde\bxi \rangle & = \left(\phi_{-(\cdot)}\bv,\partial_t \widetilde\bxi\right)_{L^2(\cS_0)} = 
\left( \bA^{-1}\bv(\cdot,\Phi_t^n(\cdot))\,,\partial_t (\phi_{-(\cdot)}\bxi)\right)_{L^2(\cS_0)}\\
&
= 
\left( \bA^{-T}\bA^{-1}\bv(\cdot,\Phi_t^n(\cdot))\, ,  \bA\partial_t (\phi_{-(\cdot)}\bxi)\right)_{L^2(\cS_0)}\\
&
= 
\left(J^{-1}( \bA \bA^{T})^{-1}\bv\,,\partial^\ast\bxi\right)_{\bf 0}\\
&= 
\big(\bv\, ,\partial^\ast(\bT\bxi)- (\partial^\ast\bT)\bxi\big)_{\bf 0}\quad\text{with}~\bT:=J^{-1}( \bA \bA^T)^{-1}\\
&= 
-\langle\partial^\ast\bv\, ,\bT\bxi\rangle- \big(\bv,w_N\kappa\,\bT\bxi+ (\partial^\ast\bT)\bxi\big)_{\bf 0}.
\end{split}
\end{equation}
Note that $\bT:T\Gamma(t)\to T\Gamma(t)$ for all $t\in[0,T]$ and from Lemma~\ref{lemS} it follows that $\bT,\bT^{-1}\in C^1(
\bar\Gs)^{3\times3}$. Hence it holds 
\[
\bT\bxi\in L^2_{H^1}\quad\text{and}\quad \|\bT\bxi\|_{\bf 1}\simeq \|\bxi\|_{\bf 1}\simeq \|\widetilde\bxi\|_{L^2([0,T],\Vd(0))}.
\]
From this, equality~\eqref{aux516} and \eqref{VHeqv} one obtains after  simple calculations,
\[
\|\partial_t(\phi_{-(.)}\bv)\|_{L^2([0,T],\Vd(0)')}\le C (\|\partial^\ast\bv\|_{L_{H^{-1}}^2}+\|\bv\|_{\bf 0})
\le C(\|\partial^\ast\bv\|_{L_{\Vd'}^2}+\|\bv\|_{\bf 0})\le C\|\bv\|_{W_\ast}.
\]
The reverse estimate $\|\partial^\ast\bv\|_{L_{\Vd'}^2}\le C \|\bv\|_{\cW(\Vd(0),\Vd(0)')}$
follows from the identity 
\[
\langle\partial^\ast\bv\, ,\bxi\rangle
= -\langle \partial_t (\phi_{-(\cdot)}\bv), \phi_{-(\cdot)}\bT^{-1}\bxi \rangle - \big(\bv,w_N\kappa\,\bxi+ (\partial^\ast\bT)\bT^{-1}\bxi\big)_{\bf 0}
\]
by similar arguments (in particular the analogue result to \eqref{VHeqv} holds for the time derivative $\partial_t$ on $\Gs_0$). 
Therefore we proved $\|\bv\|_W\simeq\|\phi_{-(.)}\bv\|_{\cW(\Vd(0),\Vd(0)')}$ and hence $\bW(\Vd,\Vd')$ and $\cW(\Vd(0),\Vd(0)')$ are homeomorphic.
\quad\end{proof}

\section{Well-posed weak formulation}\label{sec4}
In this section we introduce and analyze a  weak formulation of TSNSE \eqref{NSalt}. 
We restrict our arguments to  the solenoidal case $g=0$.
The extension of the analysis to the case $g \neq 0$ is discussed in section~\ref{sectnonsolonoidal}.  
In the weak formulation we take a solution space with only solenoidal vector fields, and thus the pressure term vanishes. The existence of a corresponding unique pressure solution is shown in section~\ref{sectpressure}.

We introduce the notation 
\begin{equation} \label{not1} \begin{split}
 a(\bu,\bv) &: = 2\mu(E_s(\bu),E_s(\bv))_{\bf 0},~ ~  
 c(\bu, \tilde \bu, \bv) := ((\nablaG \bu) \tilde \bu, \bv)_{\bf 0},\\ 
 \ell(\bu,\bv)&  := (w_N \bH \bu,\bv)_{\bf 0},
\end{split}
\end{equation}
and consider the following \emph{weak formulation of TSNSE} \eqref{NSalt} with $g=0$: For given $\blf \in L^2(\cS)^3$, with  $\blf =\blf_T$,  $\bu_0 \in \Hd(0)$,
find $\bu_T \in \bW(\Vd,\Vd')$ such that $\bu_T(0)=\bu_0$ and 
\begin{equation} \label{NSvar}
  \langle \pc \bu_T,\bv\rangle + a(\bu_T,\bv)  +
  c(\bu_T,\bu_T, \bv)  + \ell(\bu_T,\bv) =(\blf,\bv)_{\bf 0} \quad \text{for all}~~ \bv \in L_{\Vd}^2.
\end{equation}
One easily checks that any smooth solution of \eqref{NSalt} satisfies \eqref{NSvar}.

For the analysis of the weak formulation \eqref{NSvar} we apply an established approach, e.g. \cite{Temam77}. Compared to the analysis of the non-stationary Navier-Stokes equations in Euclidean domains the main differences are that we use evolving spaces as introduced above instead of the standard Bochner ones, we have a normal time derivative $\pc$ in place of the usual $\frac{d}{dt}$, and an additional curvature-dependent term $(w_N \bH \bu_T,\bv)_{\bf 0}$. 
We show the existence of a Galerkin solution,  derive a-priori bounds and based on  this show existence of a solution $\bu_T$.  We then show uniqueness of the solution with the help of Ladyzhenskaya's inequality. 

\subsubsection*{Faedo--Galerkin approximation} The space $\Vd(0)$ has a countable basis $\bpsi_1,\bpsi_2, \ldots$, which is pushed forward to a countable basis $\{\tbpsi_i\}$ of $\Vd(t)$ by letting $\tbpsi_i=\phi_t\bpsi_i$. Consider
\begin{equation} \label{Ansatz} 
 \bu_m:= \sum_{i=1}^m g_{i,m}(t) \tbpsi_i. 
\end{equation}
We determine the unknown functions $g_{i,m}$ from \eqref{Ansatz} by considering the system of ODEs  
\begin{equation} \label{R2} \begin{split}
  (\pc \bu_m,  \tbpsi_j)_{0,t} &+ 2\mu(E_s(\bu_m),E_s(\tbpsi_j))_{0,t} + ((\nablaG\bu_m) \bu_m,  \tbpsi_j )_{0,t} \\ & + (w_N \bH\bu_m,\tbpsi_j)_{0,t}=(\blf,\tbpsi_j)_{0,t} \quad \text{for all}~1 \leq j \leq m .\\
 & \bu_m(0)= \bu_{0m},
 \end{split}
\end{equation}
Here $\bu_{0m}$ is the $L^2$-orthogonal projection of $\bu_0$ on ${\rm span}\{\bpsi_1, \ldots,\bpsi_m\}$. 

\subsubsection*{A priori bounds} Assume  $\bu_m$ as in \eqref{Ansatz} satisfies \eqref{R2}.  Multiplying \eqref{R2} by $g_{j,m}(t)$ and summing over $j=1,\ldots m$, we get, using $((\nablaG\bu_m) \bu_m,  \bu_m )_{0,t}=0$,
\begin{equation} \label{R3} 
	(\pc \bu_m,  \bu_m)_{0,t} +2\mu(E_s(\bu_m),E_s(\bu_m))_{0,t} + (w_N \bH\bu_m,\bu_m)_{0,t} =(\blf,\bu_m)_{0,t},
\end{equation}
and applying integration by parts \eqref{PIid} we have
\begin{equation} \label{R4} \begin{split} 
    &\frac{d}{dt} \|\bu_m\|_{0,t}^2 + 4\mu\|E_s(\bu_m)\|_{0,t}^2 \\ & = 
		- 2 (w_N \bH\bu_m,\bu_m)_{0,t} + (w_N \kappa \bu_m, \bu_m)_{0,t} +2(\blf,\bu_m)_{0,t}. 
\end{split} \end{equation}
From this we obtain for $0<\tau \leq T$,
\begin{multline} \label{R5} 
		\|\bu_m\|_{0,\tau}^2 + 4\mu \int_0^\tau\|E_s(\bu_m)\|_{0,t}^2 \, dt 
		 \lesssim \int_0^\tau \|\bu_m\|_{0,t}^2\, dt + \int_0^\tau \| \blf \|_{0,t}^2 \, dt + \|\bu_{0m}\|_{L^2(\Gamma_0)}^2.
 \end{multline}
Here and in the remainder we write $A\lesssim B$ to denote $A\le c\,B$ with some constant $c$ which may depend on the final time $T$,  the maximum normal velocity $\|w_N\|_{L^\infty(\Gs)}$ and on smoothness properties of the space-time manifold, quantified by  $\|\bH\|_{L^\infty(\Gs)}$. Note that $\|\kappa\|_{L^\infty(\Gs)}=\|\text{tr}(\bH)\|_{L^\infty(\Gs)}\le2\|\bH\|_{L^\infty(\Gs)}$. The Gronwall lemma  and \eqref{R5} yield the a priori bound,
\begin{equation} \label{apriori1}
	\max_{0 \leq t \leq T} \|\bu_m\|_{0,t} + \|E_s(\bu_m)\|_{\bf 0}  \lesssim \| \blf \|_{\bf 0}+\|\bu_{0}\|_{L^2(\Gamma_0)} .
\end{equation}

The uniform Korn inequality  and the estimates in \eqref{R5}-\eqref{apriori1} yield the a priori estimate
\begin{equation} \label{apriori2}
	\|\bu_m\|_{\bf 1} \lesssim \| \blf \|_{\bf 0}+ \|\bu_{0}\|_{L^2(\Gamma_0)}.
\end{equation}

\subsubsection*{Existence of solution} Consider the ODEs system \eqref{R2}.
Due to Lemma~\ref{lemrelation} we have $\bP\pc(\tbpsi_i)=\partial^\ast(\tbpsi_i)- \bC  \tbpsi_i=-\bC  \tbpsi_i $, with $\bC=\bA \bP \pc  \bA^{-1}$. Thus \eqref{R2}  results in the following system for $g_{i,m}$: 
\begin{multline} \label{aux686} 
\sum_{i=1}^m\frac{dg_{i,m}(t)}{dt} (\tbpsi_i,  \tbpsi_j)_{0,t}= - \sum_{i,k=1}^m g_{i,m}(t) g_{k,m}(t) ((\nablaG\tbpsi_k) \tbpsi_i,  \tbpsi_j )_{0,t} \\ -\sum_{i=1}^m g_{i,m}(t)\Big\{2\mu(E_s(\tbpsi_i),E_s(\tbpsi_j))_{0,t} + ((w_N \bH-\bC)\tbpsi_i,\tbpsi_j)_{0,t}\Big\}+(\blf,\tbpsi_j)_{0,t}, 
\end{multline}
for $1 \leq j \leq m$. From the fact that the pushforward map $\phi_t$ is one-to-one and linear for every $t$, and $\bpsi_i$ are linear independent we infer that $\tbpsi_i$ are linear independent for every $t$ and thus the matrix $\bM(t):=(\tbpsi_i(t),\tbpsi_j(t))_{1 \leq i,j\leq m}$ is invertible for $t\in[0,T]$. Moreover,  \eqref{smooth_est} and the definition of $\tbpsi$ implies $\bM\in C^1[0,T]^{m\times m}$. Since any eigenvalue of $\bM$, denoted by $\lambda(\bM)$,  continuously depends on matrix coefficients,  the bound $\lambda(\bM)>0$ for each $t\in[0,T]$ implies $\lambda(\bM)\ge c>0$ uniformly on $[0,T]$. The uniform lower bound for the eigenvalues and the symmetry of $\bM$ ensures $\|\bM^{-1}\|_{C[0,T]}\le C$. 
Multiplying both sides of \eqref{aux686} with $\bM^{-1}$, one verifies that the Picard-Lindel\"of theorem applies. Hence a unique solution $g_{i,m}(t)$, $1 \leq i \leq m$, exists for a maximal interval $[0,t_e]$, $t_e >0$. If $t_e < T$, then  $\lim_{t \uparrow t_e}\|\bu_m(t)\|_{0,t} = \infty$, which contradicts the established bound \eqref{apriori1} with $T$ replaced by $t_e$. Hence, a unique solution $\bu_m(t)$ exists for $t \in [0,T]$.
\smallskip

 From the a priori bounds \eqref{apriori1}, \eqref{apriori2}  it follows that there is a subsequence of $(\bu_{m'})_{m'\geq 1}$ of $(\bu_m)_{m\geq 1}$ that is weak-star convergent in $L_H^{\infty}$ and weakly convergent in $L_{\Vd}^2$ to $\bu^\ast \in 
L_H^{\infty} \cap L_{\Vd}^2$. Due to the compactness of $L_{\Vd}^2 \hookrightarrow L^2_{\Hd}$ this sequence also strongly converges in $L^2_{\Hd}$. Now note that, with $\tilde \bpsi_j$, $j=1,2,\ldots,$ as above, functions $\sum_{j=1}^N g_j(t) \tbpsi_j$, $N\in \Bbb{N}$, $g_j \in C^1([0,T])$, with $g_j(T)=0$, are dense in $L_{\Vd}^2$.
We multiply \eqref{R2} with such a function $g_j$, integrate over $[0,T]$ and apply partial integration \eqref{PIid}, which yields 
 \begin{align}
 & - \big(\bu_{m'}, \pc( \tbpsi_j g_j)\big)_{\bf 0} - (\kappa w_N \bu_{m'},  \tbpsi_j g_j)_{\bf 0}   \label{R6a}\\ & + 2\mu \big(E_s(\bu_{m'}),E_s(\tbpsi_j g_j)\big)_{\bf 0} + (\nablaG\bu_{m'} \bu_{m'},  \tbpsi_j g_j )_{\bf 0}  \label{R6b}\\ & + (w_N \bH\bu_{m'},\tbpsi_jg_j)_{\bf 0} =(\bu_{0m}, \bpsi_j g_j(0))_{L^2(\Gamma_0)} + (\blf,\tbpsi_j g_j)_{\bf 0} \label{R6c}
 \end{align}
 Due to the strong convergence $\bu_{m'} \to \bu^\ast$ in $L^2_{\Hd}$ we can pass to the limit in the two terms in \eqref{R6a} and the first term in \eqref{R6c}. 
Since $\big(E_s(\cdot),E_s(\bv)\big)_{\bf 0}$ is a functional on $L_{\Vd}^2$ for any $\bv \in L_{\Vd}^2$, we  can pass to the limit in the first term in \eqref{R6b}. Using the strong convergence in $L^2_{\Hd}$ we can also pass to the limit in the second term in \eqref{R6b}, cf. \cite[Lemma 3.2]{Temam77}. By definition of $\bu_{0m}$ we have $\bu_{0m} \to \bu(0)$ strongly in $L^2(\Gamma_0)$.
Thus we get, cf.~\eqref{not1},
\begin{multline} \label{R7} 
    -(\bu^\ast,\pc(\tbpsi_j g_j))_{\bf 0} - (\kappa w_N \bu^\ast,\tbpsi_j g_j)_{\bf 0}
    =- a(\bu^\ast,\tbpsi_j g_j) \\ 
     - c(\bu^\ast,\bu^\ast,\tbpsi_j g_j)  - \ell(\bu^\ast,\tbpsi_j g_j)+(\bu_{0}, \bpsi_j g_j(0))_{L^2(\Gamma_0)} + (\blf,\tbpsi_j g_j)_{\bf 0}. 
\end{multline}
We restrict to $g_j$ with $g_j(0)=0$ and build linear combinations of \eqref{R7} to arrive at
\begin{equation} \label{R8} 
	-(\bu^\ast,\pc \bv)_{\bf 0} - (\kappa w_N \bu^\ast,\bv)_{\bf 0}
	=- a(\bu^\ast,\bv)  
	- c(\bu^\ast,\bu^\ast,\bv)  - \ell(\bu^\ast,\bv)+ (\blf,\bv)_{\bf 0} 
\end{equation}
for all $\bv=\sum_{j=1}^N\tbpsi_j g_j$. 
We estimate the nonlinear term  with the help of uniform Ladyzhenskaya inequality and \eqref{apriori1}, \eqref{apriori2}:
\begin{equation}\label{aux812}
\begin{split}
	c(\bu^\ast,\bu^\ast,\bv)&=-c(\bv, \bu^\ast,\bu^\ast)\le \int_{0}^{T}\|\nablaG \bv\|_{0,t}\|\bu^\ast\|^2_{L^4(\Gamma(t))}dt 
	\\
	&\lesssim \int_{0}^{T}\| \bv\|_{1,t}\|\bu^\ast\|_{0,t}\|\bu^\ast\|_{1,t}dt\lesssim \sup_{t\in[0,T]}\|\bu^\ast\|_{0,t} \|\bu^\ast\|_{\bf 1}\|\bv\|_{\bf 1}\\
	&\lesssim (\| \blf \|_{\bf 0}+ \|\bu_{0}\|_{L^2(\Gamma_0)})^2\|\bv\|_{\bf 1}.
\end{split}
\end{equation}
Using the above estimate and obvious continuity estimates for other terms on the right hand side in \eqref{R8} together with a density argument we conclude that 
\begin{equation}\label{apriori3}
	\|\pc \bu^\ast\|_{L_{\Vd'}^2}\lesssim F(1+F),\quad\text{with}~~  F:=\| \blf \|_{\bf 0}+ \|\bu_{0}\|_{L^2(\Gamma_0)},
\end{equation}
 hence $\bu^\ast \in \bW(\Vd,\Vd')$ and furthermore $\bu_T=\bu^\ast$ satisfies \eqref{NSvar}. 

To check that $\bu^\ast(0)=\bu_0$ holds, we apply standard arguments. Using continuity of $t \to \|\bv(t)\|_{0,t}$ for $\bv \in \bW(\Vd,\Vd')$ (\cite[Theorem2.40]{alphonse}) and density of smooth functions in $\bW(\Vd,\Vd')$ it follows that the partial integration rule \eqref{PIid} can be generalized to $\bW(\Vd,\Vd')$. Test \eqref{NSvar} with $\bv=\tbpsi_j g_j(t)$,  with $g_j(0)=1$, applying partial integration and comparing the result with \eqref{R7} we obtain $(\bu^\ast, \bpsi_j)_{L^2(\Gamma_0)}=(\bu_{0}, \bpsi_j)_{L^2(\Gamma_0)}$. Since $(\bpsi_j)_{j \in \Bbb{N}}$ is dense in $\Hd(0)$ we conclude that $\bu^\ast(0)=\bu_0$ holds.

\subsubsection*{Uniqueness of solution} We prove uniqueness of the solution using essentially the same arguments as in Euclidean space.
For the sake of presentation below, we use $\llangle\cdot,\cdot\rrangle$ to denote $L^2$ duality pairing between $\Vd(t)$ and $\Vd(t)'$   and introduce the notation for $t$-level bilinear forms, cf. \eqref{not1}:
\begin{align*}
 a_t(\bv,\bpsi) : = 2\mu(E_s(\bv),E_s(\bpsi))_{0,t}, ~  
 c_t(\bv, \tilde \bv, \bpsi) := ((\nablaG \bv) \tilde \bv, \bpsi)_{0,t},~
 \ell_t(\bv,\bpsi) := (w_N \bH \bv,\bpsi)_{0,t}.
\end{align*}
Note that $c_t(\bv, \tilde \bv, \bv)=0$ holds.
A solution $\bu_T \in \bW(\Vd,\Vd')$ of \eqref{NSvar} satisfies
\begin{multline} \label{NSvarloc}  \llangle \pc \bu_T(t),\bv(t)\rrangle + a_t(\bu_T(t),\bv(t)) + c_t(\bu_T(t),  \bu_T(t), \bv(t)) \\ + \ell_t(\bu_T(t),\bv(t))= (\blf(t),\bv(t))_{0,t} \quad \text{a.e. in $t$}~~\text{for}~\bv \in L_{\Vd}^2.
\end{multline}
The Leibniz rule \eqref{PI1} extends to  $\bv \in \bW(\Vd,\Vd')$ (cf. \cite[Lemma 3.4]{alphonse})  yielding
\[
 \frac{d}{dt} (\bv(t),\bv(t))_{0,t} = 2\llangle \pc \bv(t), \bv(t)\rrangle + \int_{\Gamma(t)} |\bv(t)|^2 w_N \kappa \, ds.    
\]
 Let  $\bu_T^1$, $\bu_T^2$ be solutions of \eqref{NSvar}  with $\bu_T^1(0)=\bu_T^2(0)=\bu_0$. Letting  $\bpsi:=\bu_T^1-\bu_T^2$ and using $c_t(\bpsi,\bu_T^1,\bpsi)=0$ we compute, with $C_1:=\|w_N \kappa\|_{C(\cS)}$, 
\begin{align*}
  \frac{d}{dt} \|\bpsi(t)\|_{0,t}^2 &+ 2 a_t(\bpsi(t),\bpsi(t)) \\ & =2\llangle \pc \bpsi(t), \bpsi(t)\rrangle + \int_{\Gamma(t)} |\bpsi(t)|^2 w_N \kappa \, ds + 2 a_t(\bpsi(t),\bpsi(t)) \\
 & \leq \phantom{-} 2 \llangle \pc \bu_T^1(t), \bpsi(t)\rrangle +2a_t(\bu_T^1(t),\bpsi(t)) \\ & \quad - 
   2 \llangle \pc \bu_T^2(t), \bpsi(t)\rrangle -2a_t(\bu_T^2(t),\bpsi(t)) + C_1 \|\bpsi(t)\|_{0,t}^2  \\
   & = -2 c_t(\bu_T^1(t),  \bu_T^1(t), \bpsi(t)) -2  \ell_t(\bu_T^1(t),\bpsi(t)) \\
   &  \quad +2 c_t(\bu_T^2(t),  \bu_T^2(t), \bpsi(t)) -2 \ell_t(\bu_T^2(t),\bpsi(t))+ C_1 \|\bpsi(t)\|_{0,t}^2  \\
   & = - 2 c_t(\bu_T^2(t), \bpsi(t),\bpsi(t)) - 2\ell_t(\bpsi(t),\bpsi(t)) + C_1 \|\bpsi(t)\|_{0,t}^2. 
\end{align*}
We have $|\ell_t(\bpsi(t),\bpsi(t))|\leq  \|w_N \bH\|_{C(\cS)}\|\bpsi(t)\|_{0,t}^2$. For the other terms on the right hand side above we use \eqref{Lady} and the Korn inequality \eqref{Korn}  to estimate
\begin{align*}
 \frac{d}{dt} \|\bpsi(t)\|_{0,t}^2 &+ 2 a_t(\bpsi(t),\bpsi(t)) \\ & \leq 
 C \|\bu_T^2\|_{1,t} \|\bpsi(t)\|_{L^4(\Gamma(t))}^2 +  C \|\bpsi(t)\|_{0,t}^2 \\
 & \leq C \|\bu_T^2\|_{1,t}\|\bpsi(t)\|_{0,t}\|\bpsi\|_{1,t} + C \|\bpsi(t)\|_{0,t}^2 \\
  & \leq C \|\bu_T^2\|_{1,t}\|\bpsi(t)\|_{0,t} \big( \|\bpsi(t)\|_{0,t} +a_t(\bpsi(t),\bpsi(t))^\frac12 \big)+ C \|\bpsi(t)\|_{0,t}^2 \\
  & \leq C \big(\|\bu_T^2\|_{1,t}+\|\bu_T^2\|_{1,t}^2 \big)\|\bpsi(t)\|_{0,t}^2 +  2 a_t(\bpsi(t),\bpsi(t)),
\end{align*}
with a suitable constant $C$ independent of $t \in [0,T]$ and of $\bu_T^1,\, \bu_T^2$. 
Thus we get 
\[
  \frac{d}{dt} \|\bpsi(t)\|_{0,t}^2 \leq f_u(t)\|\bpsi(t)\|_{0,t}^2,  \quad f_u(t):= C \big(\|\bu_T^2\|_{1,t}+\|\bu_T^2\|_{1,t}^2\big). 
\]
Now, $\bu_T^2\in L_{\Vd}^2$ implies that $\int_0^T f_u(s)\,ds$ is bounded and so the Gronwall inequality together with 
    $\|\bpsi(0)\|_{L^2(\Gamma_0)}=0$ yields $\bpsi(t)=0$ for $t\in [0,T]$ and thus the uniqueness result holds. 
%

Summarizing we proved  the following main well-posedness result.
\begin{theorem}  The weak formulation \eqref{NSvar}  of the TSNSE  has a unique solution $\bu_T\in \bW(\Vd,\Vd')$. The solution satisfies
\begin{equation}\label{est:apriori}
	\|\bu_T\|_W\le C (1+F)F, \quad\text{with}~~  F:=\| \blf \|_{\bf 0}+ \|\bu_{0}\|_{L^2(\Gamma_0)}.
\end{equation}
\end{theorem}
%
\subsection{Surface pressure} \label{sectpressure} 
For $\bv\in L^2_{\Vd}$, \eqref{weak} defines  $\pc \bv$ as a functional on $\widehat{\cD}_0$.
The density of $\widehat{\cD}_0$ in $L^2_{H^1}$ and the density of ${\cD}_0\subset\widehat{\cD}_0$ in $L^2_{\Vd}$ is used to define the bounded linear functionals  $\pc \bv\in L^2_{H^{-1}}$ and $\pc \bv\in L^2_{\Vd'}$, respectively.
The following equivalence holds:
\begin{equation}\label{VHeqv2}
	\pc \bv\in L^2_{\Vd'}\quad\Longleftrightarrow\quad \pc \bv\in L^2_{H^{-1}}, \quad \bv\in L^2_{\Vd}.
\end{equation}
Implication ``$\Leftarrow$'' in  \eqref{VHeqv2} is trivial since $\Vd\subset H^{1}$. The  ``$\Rightarrow$'' implication  follows from \eqref{VHeqv} and Lemma~\ref{lemrelation}.

Below we introduce a weak formulation of  TSNSE  on the velocity space 
\[
\bW(H^1,H^{-1})=\{\bv\in L^2_{H^1}\,:\, \pc\bv\in  L^2_{H^{-1}}\}\quad\text{with}~
(\cdot,\cdot)_{W(H^1,H^{-1})}=(\cdot,\cdot)_{\bf 1}+(\cdot,\cdot)_{L^2_{H^{-1}}},
\]
 with a pressure unknown $\pi \in L^2(\cS)$. One checks that $\bW(H^1,H^{-1})$ is a Hilbert space by the same arguments as for $\bW(\Vd,\Vd')$.
Consider the following mixed formulation of  TSNSE, which relates to the well-posed weak formulation \eqref{NSvar}: For given $\blf \in L^2(\cS)^3$, with  $\blf =\blf_T$,  $\bu_0 \in \Hd(0)$,
find $\bu_T \in \bW(H^1,H^{-1})$ and $\pi \in L^2(\cS)$,  with $\int_{\Gamma(t)} \pi \, ds =0$ a.e. $t \in [0,T]$, such that $\bu_T(0)=\bu_0$ and 
\begin{equation} \label{NSvar2} \begin{split}
  \langle \pc \bu_T,\bv\rangle + a(\bu_T,\bv)  +
  c(\bu_T,\bu_T, \bv)  + \ell(\bu_T,\bv) + (\pi, \divG \bv)_{\bf 0} & =(\blf,\bv)_{\bf 0}  \\
   (q, \divG \bu_T)_{\bf 0} & = 0
\end{split}\end{equation}
for all $\bv \in L_{H^1}^2$, $q \in L^2(\cS)$.
\smallskip

\begin{theorem}
 The problem \eqref{NSvar2} has a unique solution $(\bu_T,\pi)$. The velocity solution $\bu_T$ is also the unique solution of \eqref{NSvar}. Furthermore, with $F:=\| \blf \|_{\bf 0}+ \|\bu_{0}\|_{L^2(\Gamma_0)} $ the following estimate holds
 \begin{align} \|\bu_T\|_W +\|\pi\|_{\bf 0} \le C (1+F)F. \label{RR1}
 \end{align}
\end{theorem}
\begin{proof}
 Let $\bu_T \in \bW(\Vd,\Vd')$ be the solution of \eqref{NSvar}. Define
 \[
 \widetilde{\blf}(\bv):= \langle \pc \bu_T,\bv\rangle + a(\bu_T,\bv)  +
  c(\bu_T,\bu_T, \bv)  + \ell(\bu_T,\bv) - (\blf,\bv)_{\bf 0}, \quad \bv \in L_{H^1}^2.
 \]
Using \eqref{VHeqv2} and straightforward estimates we obtain $\widetilde{\blf} \in L^2_{H^{-1}}$.\footnote{To see $\bv \to c(\bu_T,\bu_T, \bv) \in L^2_{H^{-1}}$, one uses the same arguments
as in \eqref{aux812}.} 
We use the standard argument (e.g. Remark~I.1.9 in \cite{Temam77}) that for every $t\in[0,T]$ estimate \eqref{infsup} implies that 
 $\nabla_\Gamma\in\mathcal{L}(L^2(\Gamma(t)),H^{-1}(t))$ has a closed range $R(\nabla_\Gamma)$ in $H^{-1}(t)$ and so 
 \[
 R(\nabla_\Gamma) = \ker(\nabla_\Gamma^\ast)^\perp,\quad\text{with}~\ker(\nabla_\Gamma^\ast)=\Vd(t).
 \]
Note that  $\widetilde{\blf}(t)$ is an element of $H^{-1}(t)$ for a.e. $t\in[0,T]$ and, since $\bu_T$ is the solution of \eqref{NSvar},
$\llangle \widetilde{\blf}(t),\bv\rrangle=0$ for all $\bv\in \Vd(t)$.  Hence,  $\widetilde{\blf}(t)\in R(\nabla_\Gamma)$ which means
\[
\widetilde{\blf}(t)=\nabla_\Gamma\pi(t)\quad \text{for some}~\pi(t)\in L^2(\Gamma(t))~~\text{for a.e.}~t\in[0,T].
\]
We take $\pi(t)$ such that $\int_{\Gamma(t)}\pi(t)=0$ holds. To see that $t\to\|\pi(t)\|_{0,t}$ is  measurable, we note the following. Let $\phi= \phi(t) \in H^2(\Gamma(t))$, with $\int_{\Gamma(t)} \phi \, ds =0$ be the solution of the Laplace-Beltrami equation $-\Delta_{\Gamma} \phi = \pi(t)$ on $\Gamma(t)$.  We then have $\|\pi(t)\|_{0,t}^2= \int_{\Gamma(t)} \nabla_{\Gamma} \phi(t) \cdot \nabla_{\Gamma} \pi(t) \, ds = \int_{\Gamma(t)}\nabla_{\Gamma} 
\phi(t) \cdot\widetilde{\blf}(t)\, ds$. From $\widetilde{\blf} \in L^2_{H^{-1}}$ and $\nabla_{\Gamma} \phi(t) \in H^1(t)$ we conclude that $t\to\|\pi(t)\|_{0,t}$ is  measurable. 
From \eqref{NSvar2} we get, with notation as in \eqref{not1},
\begin{align*} (\pi, \divG \bv)_{0,t} = &(\blf(t),\bv(t))_{0,t}-\llangle \pc \bu_T(t),\bv(t)\rrangle   - a_t(\bu_T(t),\bv(t)) \\ & - c_t(\bu_T(t),  \bu_T(t), \bv(t)) -  \ell_t(\bu_T(t),\bv(t))  \quad \text{a.e. in $t$}~~\text{for}~\bv \in L_{\Vd}^2.
 \end{align*}
Using the uniform inf-sup estimate, cf. Lemma~\ref{leminfsup}, we get
\[
 \|\pi(t)\|_{0,t} \leq C \big( \|\blf(t)\|_{0,t} +\|\pc \bu_T\|_{\Vd(t)'}+ \|\bu_T(t)\|_{1,t}(1+\|\bu_T(t)\|_{0,t})\big),
\]
with a constant $C$ independent of $t$. 
Hence, $\pi \in L^2(\cS)$ holds.  The estimate for velocity in  \eqref{RR1} is the same as in Theorem~\ref{est:apriori}.
Note that $\max_{t \in [0,T]} \|\bu_T(t)\|_{0,t} \lesssim F$ holds, cf. \eqref{apriori1}. Using this and the velocity estimate we obtain the bound for the pressure in \eqref{RR1}. Uniqueness of $\bu_T$ follows by restricting to $\bv \in L_{\Vd}^2$ in \eqref{NSvar2} and using the fact that \eqref{NSvar} has a unique solution. Uniqueness of $\pi$ is easily derived using the inf-sup property.
\end{proof}

\subsection{Energy balance} \label{sectenergy}
Multiplying \eqref{NSalt} by $\bu_T$, integrating over $\Gamma(t)$ and using \eqref{PI1}, we obtain, for a  smooth solution,   the  energy balance of the system  at any $t\in(0,T)$, 
\begin{equation}\label{eq:Energy}
		\frac12\frac{d}{dt} \|\bu_T\|_{0,t}^2 + 2\mu\|E_s(\bu_T)\|_{0,t}^2 +  (w_N (\bH-\tfrac12\kappa\bI)\bu_T,\bu_T)_{0,t} = (\blf,\bu_T)_{0,t}.
\end{equation}   
Next we comment on the contribution of the third term in \eqref{eq:Energy}, which appears if the surface is both evolving and non-flat.
First we note that $\bH-\tfrac12\kappa\bI=\bH-\tfrac12\kappa\bP$ on $T\Gamma(t)$ and  $\bH-\tfrac12\kappa\bP=\bH-\tfrac12\mbox{tr}(\bH)\bP$. Since $\mbox{tr}(\bP)=2$ we get  $\mbox{tr}(\bH-\tfrac12\kappa\bP)=0$. This implies that the symmetric  matrix $\bH-\tfrac12\kappa\bP$
has real eigenvalues $\{0,\lambda, -\lambda\}$ with the $0$ eigenvalue corresponding to vectors normal to $\Gamma(t)$. Take a fixed point $\bx$ on the surface $\Gamma(t)$  with $w_N(\bx) \neq 0$.
Denote by $\kappa_1$ and $\kappa_2$ the two principle curvatures of $\Gamma(t)$. For the eigenvalue of $\bH(\bx)-\tfrac12\kappa(\bx)\bP(\bx)$ 
 we have $\lambda(\bx)=\tfrac12\big(\kappa_1(\bx)-\kappa_2(\bx)\big)$. Therefore $w_N(\bx) (\bH(\bx)-\tfrac12\kappa(\bx)\bP(\bx))=0$ iff $\kappa_1(\bx)=\kappa_2(\bx)$ holds,  and it is indefinite otherwise. In the latter case increase or decrease
of kinetic  energy due to this term depends on the alignment of the flow with the principle  directions and the sign of $w_N$.  

\subsection{Non-solonoidal solution} \label{sectnonsolonoidal}
The tangential surface Navier--Stokes system \eqref{NSalt} admits non-solonoidal solutions with $\divG\bu_T=g$, where $g=-w_N\kappa$, $\int_{\Gamma(t)} g\,ds=0$ for $t\in[0,T]$, is defined by the surface geometry and evolution.  We outline how the analysis for the solonoidal case presented above can be extended to such a problem. We assume that $g:\, \Gs \to \Bbb{R} $ is sufficiently regular. Let $\phi(t,\bx)$ be the unique solution of the Laplace-Beltrami equation $ \Delta_{\Gamma(t)} \phi = g$, $\int_{\Gamma(t)} \phi \, ds =0$,  and define $\widetilde \bu_T:= \nablaG \phi$. Then $ (\bu_T, \pi)$ solves  \eqref{NSalt} iff $\bhu_T=\bu_T-\widetilde \bu_T$ and $\pi$ solve the 
system 
\begin{equation}\label{NSaltg}
	\left\{\begin{aligned}
		\bP \pc\bhu_T +  (\nablaG \bhu_T) \bhu_T+  (\nablaG \widetilde \bu_T) \bhu_T+  (\nablaG \bhu_T) \widetilde \bu_T \hskip5ex &
		\\
		+ w_N \bH \bhu_T- 2\mu  \bP \divG E_s(\bhu_T) +\nabla_\Gamma\pi&=\widehat{\blf} \\
		\divG \bhu_T   &= 0,
	\end{aligned}\right.
\end{equation} 
with
\[
\widehat{\blf}=\blf - \big(\bP \pc\widetilde \bu_T +  (\nablaG \widetilde \bu_T) \widetilde \bu_T+ w_N \bH \widetilde \bu_T- 2\mu  \bP \divG E_s(\widetilde \bu_T)\big).
\]
The two additional  terms $(\nablaG \widetilde \bu_T) \bhu_T$ and  $(\nablaG \bhu_T) \widetilde \bu_T$ in the momentum equation in \eqref{NSaltg} are linear with respect to the unknown velocity field $\bhu_T$ and  can be treated very similar to the zero order term $w_N \bH \bhu_T$. 
The necessary regularity of $\widehat{\blf}$ can be established using the smoothness of $g$ and $\Gs$. We skip working out further details.

\section{Discretization method} \label{s:discretization}
As discussed in the introduction, only very few papers are available in which finite element discretization methods for vector- or tensor valued surface PDEs, such as the surface Navier-Stokes equations, on evolving surfaces are treated.
In this section we present a discretization  method for the TSNSE \eqref{NSalt}. The method is based on a combination of an implicit time stepping scheme with a TraceFEM for discretization in space. 
 The general idea behind TraceFEM is to use standard time-independent (bulk) finite element spaces  to approximate surface quantities.  The method is based on tangential calculus in the ambient  space $\mathbb{R}^3 \supset \Gamma(t)$. For scalar PDEs on evolving surfaces, a space--time  variant of TraceFEM is treated in \cite{olshanskii2014eulerian}. A  finite difference (FD) in time -- FEM in space variant for PDEs on time-dependent surfaces is treated in  \cite{lehrenfeld2018stabilized} (scalar problems) and in time-dependent volumetric domains in \cite{LehrenfeldOlshanskii2019} (scalar equations) and \cite{Lehrenfeldetal2021} (Stokes problem). Compared to the space--time variant the FD--FEM approach is more flexible in terms of implementation and the choice of elements. Below we explain this FD--FEM approach applied to the TSNSE. We start with the numerical treatment of the system's evolution in time.

\subsection{Time-stepping scheme}
Consider uniformly distributed time nodes  $t_n=n\Delta t$, $n=0,\dots,N$, with the  time step $\Delta t=T/N$. We assume that the time step $\Delta t$ is sufficiently small such that
\begin{equation}\label{ass1}
	\Gamma(t_n)\subset\mathcal{O}( \Gamma(t_{n-1})),\quad~n=1,\dots,N, ~ 
\end{equation}
with $\mathcal{O}( \Gamma(t))$ a neighborhood of $ \Gamma(t)$ where a smooth  extension  of surface quantities on $\Gamma(t)$  is well defined.  Assuming a smooth extension $\bu_T^e$,  we rewrite the normal time derivative $\pc$ used in \eqref{NSalt} in terms of standard time and space derivatives:
\begin{equation} \label{idpc}
	\bP \pc \bu_T +(\gradG \bu_T) \bu_T = \bP\big( \frac{\partial \bu_T^e}{\partial t} + (\nabla \bu_T^e) \bw_N + (\nabla \bu_T^e) \bu_T\big) = \bP\big( \frac{\partial \bu_T^e}{\partial t} +  (\nabla \bu_T^e) (\bw_N+\bu_T) \big).
\end{equation}
On $\Gamma(t_n)$ the time derivative term is approximated by
\[
\bP \frac{ \partial \bu_T^e}{\partial t} \approx \frac{\bu_T(t_n)- \bP(t_n) \bu_T(t_{n-1})^e}{\Delta t}.
\]
Note that due to \eqref{ass1} $\bu_T(t_{n-1})^e$ is defined on $\Gamma(t_n)$. 
The normal surface velocity  $\bw_N$ is known, so a natural linearization  of the nonlinear term in \eqref{idpc}    is given by 
\[
\bP \nabla \bu_T^e (\bw_N+\bu_T) \approx \bP(t_n)  \nabla \bu_T(t_n)^e (\bw_N(t_n) + \bu_T(t_{n-1})^e)\quad \text{on} ~\Gamma(t_n).
\]
The finite difference approximations above need extensions of quantities defined on $\Gamma(t_j)$ to $\mathcal{O}( \Gamma(t_{j}))$.  It  is natural to consider a normal extension, which  can be characterized as follows. 
Let $\bn=\nabla\phi$ in $\mathcal{O}(\Gamma(t_j))$, where $\phi$ is the signed distance function for $\Gamma(t_j)$, and $g$ a function defined on $\Gamma(t_j)$. The normal extension $g^e$ of $g$ is such that $g^e=g$ on $\Gamma(t_j)$ and    
\begin{equation}
	\bn\cdot\nabla g^e =0\quad\text{in}~~\mathcal{O}(\Gamma(t_j)).
	\label{e:nExt}
\end{equation}
For practical purposes,  $\phi$ can be a 
smooth level set function for $\Gamma$ rather than a signed distance. In this case, the vector field $\bn=\nabla\phi/|\nabla\phi|$ is normal on $\Gamma$ but defines quasi-normal directions in a neighborhood. Extension of the velocity field along    quasi-normal directions is equally admissible.   
We assume that at $t=0$ an extension  $\bu_T(0)^e$ on $\mathcal{O}(\Gamma_0)$ solving \eqref{e:nExt} is given. 
We use the notation $\bu_T^j$  and $p^j$ for an approximation of $\bu_T(t_j)^e$ and  $p(t_j)$, respectively.  Based on the approximations above and  \eqref{e:nExt} consider the following time discretization method for \eqref{NSalt}. Given $\bu_T^0=\bu_T(0)^e$, for $n=1,\ldots N$, find $\bu^n_T$, defined on $\mathcal{O}(\Gamma(t_n))$ and tangential to $\Gamma(t_n)$, i.e. $(\bu^n_T\cdot\bn)|_{\Gamma(t_n)}=0$,  and $p^n$ defined on $\Gamma(t_n)$ such that:
\begin{align}
	\label{NS_FD}
	  \left\{\begin{aligned}
		\frac{\bu^n_T-\bP\bu^{n-1}_T}{\Delta t}  + \bP\nabla\bu_T^n(\bw_N^n +\bu_T^{n-1}) \qquad \qquad\qquad \quad & \\
		 +  w_N^n \bH \bu_T^n- 2\mu  \bP \divG E_s(\bu_T^n)  +\nabla_\Gamma p^n & =\blf^n\\
		\divG \bu_T^n   &= g^n\\
	\end{aligned}\right.\quad&\text{on}~~\Gamma(t_n),\\
		 \bn\cdot\nabla\bu^n_T  =0~\quad&\text{in}~~\mathcal{O}(\Gamma(t_n)),\label{e:nExt1}
\end{align}
with $w_N^n:=w_N(t_n)$, $\bw_N^n:=\bw_N(t_n)$, $\blf^n:=\blf(t_n)$, $g^n:=g(t_n)$.
Geometric information in \eqref{NS_FD} is taken for $\Gamma(t_n)$, i.e. $\bP=\bP(t_n)$, $\bH=\bH(t_n)$. For space discretization, the stationary linearized surface PDE in \eqref{NS_FD} can  be treated using a variational approach known from the literature \cite{Jankuhn1,jankuhn2019higher}, in which the tangential constraint for the solution $\bu_T^n$ is relaxed using  a penalty approach.  This technique is now outlined.  
We set  $\bc:=\bw_N^n +\bu_T^{n-1}$, $\Gn:=\Gamma(t_n)$ and introduce the following  bilinear forms on $\Gn$, with arguments $\bu,\bv$,  vector functions on $\Gn$ that are not necessarily tangential:
\begin{align}
	A(\bu,\bv)& =\frac{1}{\Delta t} \int_\Gn  \bu \cdot \bP \bv \, ds+\int_\Gn  \bv \cdot \bP (\nabla \bu) \bc \, ds +  \int_\Gn w_N^n \bu^T \bH \bv \, ds     \nonumber\\
	& + 2 \mu \int_\Gn E_s(\bP\bu):E_s(\bP\bv) \, ds + \tau \int_{\Gn} u_N v_N\, ds,  \label{defA} \\
	B(\bu,p) & = - \int_\Gn p \, \divG \bP \bu \, ds ,\label{defB}
\end{align}
where $\tau >0$ is a penalty parameter. We introduce two Hilbert  spaces  
\[
\begin{split}
L^2_0(\Gn)&:= \{\, p \in L^2(\Gn)~|~\int_\Gn p \, ds =0\,\},~~ \text{and}\\
\bV_\ast&:= \{ \, \bv \in L^2(\Gn)^3~|~ \bv_T \in H^1(\Gn)^3,~v_N\in L^2(\Gn)\,\},
\end{split}
\] with the norm $\|\bv\|_{V_\ast}^2= \|\bv_T\|_{H^1(\Gn)}^2 +\|v_N\|_{L^2(\Gn)}^2$. 
A variational formulation corresponding to \eqref{NS_FD} is as follows: Find $\bu_\ast \in \bV_\ast$, $p \in L^2_0(\Gn)$ such that 
\begin{equation} \label{Oseen}
	\begin{split}
		A(\bu_\ast,\bv) + B(\bv,p) & = \int_{\Gn} \tilde{\blf} \cdot \bv \, ds \quad \text{for all}~ \bv \in \bV_\ast \\
		B(\bu_\ast,q) & = - \int_{\Gn} g^n q \, ds\quad  \text{for all} ~ q \in L_0^2(\Gn),
	\end{split}
\end{equation}
with $\tilde{\blf} := \blf + \frac{1}{\Delta t} \bP \bu_T^{n-1}$. 
This variational formulation is consistent in the sense that if $(\bu_T^n,p^n)$ is a strong solution of  \eqref{NS_FD} then $(\bu_\ast,p)=(\bu_T^n,p^n)$ solves \eqref{Oseen}. 
Using the Korn type inequality \eqref{Korn} and an inf-sup property of $B(\cdot,\cdot)$ it can be shown that for  $\Delta t$ sufficiently small and $\tau$ sufficiently large (but independent of $\Delta t$) the problem \eqref{Oseen} is well-posed and its unique solution $\bu_\ast$ satisfies $\bu_\ast\cdot\bn=0$, cf. \cite{Jankuhn1} for a precise analysis.  Therefore, for such $\Delta t$ and $\tau$ eq. \eqref{Oseen} is a  well-posed weak formulation of \eqref{NS_FD}. For a finite element method introduced later it is important that  the space $\bV_\ast$ admits vector functions not necessarily tangential to $\Gn$. The solution $\bu$ of \eqref{Oseen} is defined only on $\Gn$ and we do not specify an extension as in \eqref{e:nExt1}, yet. Such an extension will be determined in  the finite element  method, as explained in the next section.
\begin{remark} \label{remReform} \rm
	In the practical implementation of a finite element method for \eqref{Oseen} the surface $\Gamma(t_n)$ will be approximated by a piecewise planar surface $\Gamma_h$ and the corresponding projection operator $\bP_h$ has discontinuities across boundaries between different planar segments of this approximate surface. This causes difficulties for the terms in the bilinear forms $A(\cdot,\cdot)$, $B(\cdot,\cdot)$ where derivatives of $\bP_h$  are involved. These can be avoided by eliminating these derivatives as follows. For $p \in H^1(\Gn)$ we have $B(\bu,p)= \int_\Gn \nablaG p \cdot \bP \bu \, ds=\int_\Gn \nablaG p \cdot  \bu \, ds $, which elimates derivatives of $\bP$. For the  bilinear form $A(\cdot,\cdot)$ we can use the relation $\nablaG (\bP \bu)= \nablaG \bu - u_N \bH$ and thus $E_s(\bP \bu)=E_s(\bu) - u_N \bH$.
\end{remark}

\subsection{Finite element method} \label{s:fulldisc}
We now explain the spatial discretization of \eqref{Oseen}.
Consider  a fixed polygonal domain  $\Omega \subset \R^3$ that strictly contains $\Gamma(t)$ for all $t\in(0,T)$.
Let $\{\T_h\}_{h>0}$ be  a family of shape-regular consistent triangulations of the bulk domain $\Omega$, with $\max\limits_{K\in\T_h}\mbox{diam}(K) \le h$. Corresponding to the bulk triangulation we define a standard finite element space of piecewise polynomial continuous functions of a fixed degree $k\ge1$:
\begin{equation} \label{eq:Vh}
	V_h^k=\{v_h\in C(\Omega)\,:\, v_h\in P_k(K),~~ \forall K\in \mathcal{T}_h\}.
\end{equation}
The bulk velocity and pressure finite element spaces are standard Taylor--Hood spaces:
\begin{equation*} 
	\bU_h \coloneqq  (V_h^{2})^3, \quad Q_h \coloneqq  V_h^1. 
\end{equation*}

For efficiency reasons, we use an extension not in the given ($h$ and $\Delta t$-independent) neighborhood $\mathcal{O}(\Gamma(t_n))$ of $\Gn=\Gamma(t^n)$ but in  a  \emph{narrow band} around   $\Gn$. This $\Delta t$-dependent narrow band consists   of all tetrahedra within a $\delta_n$ distance from the surface, with
\begin{equation} \label{e:delta}
	\delta_n := c_\delta \sup_{t\in(t_{n-1},t_{n})}\| w_N\|_{L^\infty(\Gamma(t))}~\Delta t
\end{equation}
and $c_\delta\geq 1 $, an $O(1)$ mesh-independent parameter. More precisely, we define the mesh-dependent narrow band 
\begin{equation*}
	{\mathcal{O}}_{\Delta t}(\Gn):= {\bigcup}\left\{\overline{K}\,:\,  K\in\mathcal{T}_h\,:\mbox{dist}(\bx,\Gn) \le\delta_n  \text{ for some } \bx \in K\right\}.
\end{equation*}
We also need a subdomain of ${\mathcal{O}}_{\Delta t}(\Gn)$  only consisting of tetrahedra intersected by $\Gn$,
\begin{equation*}
	\omega^{n}_{\Gamma}:=
	{\bigcup}\left\{ \overline{K}\in\mathcal{T}_h\,:\, K\cap\Gn\neq\emptyset\right\}.
\end{equation*}
In a time step from $t_{n-1}$ to $t_n$, the surface  may move up to $\Delta t\sup\limits_{t\in(t_{n-1},t_{n})}\| w_N\|_{L^\infty(\Gamma(t))}$ distance in normal direction, which is thus the maximum distance from $\Gn$ to $\Gamma_{n-1}$.
Therefore, $c_\delta$ in \eqref{e:delta} can be taken  sufficiently large, but independent of $h$, such that
\begin{equation}\label{cond1} 
	\omega^{n}_{\Gamma}\subset{\mathcal{O}}_{\Delta t}(\Gamma_{n-1}).
\end{equation}
This condition is the discrete analog of \eqref{ass1} and it is essential for the well-posedness of the finite element problem at time step $n$.

Next we define finite element spaces for velocity and pressure as restrictions to the narrow band ${\mathcal{O}}_{\Delta t}(\Gn)$ of the time-independent bulk spaces $\bU_h$ and $Q_h$:
\begin{equation}\label{eq:testtrial}
 \bU_h^n: =\{\, {\bv}|_{{\mathcal{O}}_{\Delta t}(\Gn)}~|~\bv \in\bU_h\, \},\quad Q_h^n:=\{\, {q}|_{{\mathcal{O}}_{\Delta t}(\Gn)}~|~q \in Q_h\, \}.  
\end{equation}
Denote by $I_h(\bv)\in\bU_h^n$ the Lagrange interpolation of $\bv\in C({\mathcal{O}}_{\Delta t}(\Gn))^3$.
Our finite element formulation is based on  formulation \eqref{Oseen}. Recall that in \eqref{Oseen} we do not require $\bu_\ast$ to be tangential to $\Gamma(t)$. The tangential condition is weakly enforced by the penalty term in \eqref{defA} with penalty parameter $\tau$. Such a penalty approach is often used in finite element methods for vector values surface PDEs~ \cite{hansbo2016analysis,jankuhn2019higher,olshanskii2018finite,olshanskii2019penalty}. In the discretization in addition to this penalty term we include two  volume terms with integrals over $\omega_\Gamma^n$ and ${{\mathcal{O}}_{\Delta t}(\Gn)}$. The discrete problem is as follows: For  given $\bu_h^{n-1} \in \bU_h^{n-1}$  and ${\bc}^{n-1}_h={\bu}^{n-1}_h+ I_h(w_N^e(t_n)\bn)$ find  $\bu_h^n\in \bU_h^n$, $p_h^n\in Q_h^n$,   satisfying
\begin{align}
	\int_{\Gn}\left(\frac{\bu^n_h-\bu^{n-1}_h}{\Delta t}+ (\nabla\bu_h^n){\bc}^{n-1}_h + u_N^n \bH \bu_h^n\right)\cdot\bP\bv_h\,ds  &  \nonumber \\
	+ 2\mu \int_{\Gn} E_s(\bP\bu_h^n):E_s(\bP\bv_h)\,ds  
	+ \tau\int_{\Gn}  (\bn\cdot\bu^n_h)(\bn\cdot\bv_h)\,ds & \label{e:FEM} \\  
	+ \int_{\Gn} \nablaG p^n \bv_h\,ds+ \rho_u\int_{{\mathcal{O}}_{\Delta t}(\Gn)}(\bn\cdot\nabla \bu^{n}_h)(\bn\cdot\nabla \bv_h)\,\diff{\vect x} & = \int_{\Gn} \blf^n\bv_h\,ds \quad \forall~ \bv_h\in \bU_h^n \nonumber \\
	- \int_{\Gn} \nablaG q\, \bu_h^n\,ds 
	+ \rho_p  \int_{\omega^{n}_{\Gamma}} (\bn \cdot \nabla p_h^n)  (\bn \cdot\nabla q_h) \,  \diff{\vect x} &=\int_{\Gn}g^nq_h\,ds \quad \forall ~q_h\in Q_h^n, \nonumber
 \end{align} 	
 for $n=1,\dots,N$. The term $\int_{{\mathcal{O}}_{\Delta t}(\Gn)}(\bn\cdot\nabla \bu^{n}_h)(\bn\cdot\nabla \bv_h)\,\diff{\vect x}$, with a parameter $\rho_u$, is included for two reasons. Firstly, this term is often used in TraceFEM to improve the conditioning of the resulting stiffness matrix, cf. e.g.~\cite{burmanembedded}. Secondly, this volume term weakly enforces the extension condition \eqref{e:nExt1} with
$\mathcal{O}(\Gamma(t_n))$ replaced by ${\mathcal{O}}_{\Delta t}(\Gn)$. In particular, at time $t_n$ a well-conditioned algebraic system is solved for all discrete velocity degrees of freedom in the neighborhood ${\mathcal{O}}_{\Delta t}(\Gn)$; we refer to \cite{lehrenfeld2018stabilized} for a stability and convergence analysis of such an extension procedure for a scalar surface equation. 
The  volume term in the pressure equation is added for the purpose of numerical  stabilization of pressure~\cite{olshanskii2021inf}. The formulation \eqref{e:FEM} is consistent in the sense that the equations hold  if the solution of \eqref{NS_FD}, extended along normal directions, is substituted instead of $u_h^n$ and $p_h^n$.   Penalty and stabilization parameters are set following the error analysis in~\cite{olshanskii2021inf}:
\[
\tau=h^{-2},\quad \rho_u=h^{-1},\quad \rho_p=h.
\]
In practice, $\Gn$, $n=1,2,\dots$, is replaced by a sufficiently accurate approximation $\Gamma_h^n$ in such a way that integrals over $\Gamma_h^n$ can be computed accurately and efficiently. Other geometric quantities, i.e. $\bn$, $\bH$ and $\bP$, are also replaced by sufficiently accurate approximations. The derivatives of projected fields, i.e. $E_s(\bP \bu_h^n)$ and $E_s(\bP \bv_h)$, are handled as discussed in Remark~\ref{remReform}. For the surface Stokes problem discretized by the trace  $\vect P_{k+1}$--$P_k$, $k\ge1$, elements, the introduced geometric error is analyzed in~\cite{jankuhn2020error}. Below we will use the lowest order trace Taylor-Hood pair  $\vect P_{2}$--$P_1$. 
An approximation $\Gamma_h^n$ that is piece-wise planar with respect to $\cT_h$ leads to an  $O(h^2)$ geometric error. This geometric error order is suboptimal given the interpolation order of the Taylor--Hood pair $\vect P_{2}$--$P_1$. This suboptimality can be overcome  by the isoparametric TraceFEM~\cite{grande2018analysis}. For numerical results in this paper we use the following less efficient but simpler  approach.  For the geometry approximation (only) we construct a piece-wise planar $\Gamma_h^n$  with respect to a local refinement of each tetrahedron from $\omega^{n}_{\Gamma}$. The number of local refinement levels is chosen sufficiently large to restore the optimal $O(h^3)$ convergence. Note that this local refinement only influences the surface approximation and not the finite element spaces used.  

Finally we note that the use of BDF2 instead of implicit Euler in the implicit time stepping scheme leads to obvious modifications of \eqref{e:FEM}. In the experiments in the next section we used this second order in time  variant of \eqref{e:FEM}.

\section{Numerical examples}\label{s_num}

For discretization,  an initial triangulation $\T_{h_0}$  was build by  dividing $\Omega=(-\frac53,\frac53)^3$ into $2^3$ cubes and further splitting each cube into 6 tetrahedra with  $h_0 = \frac53$. Further, the mesh is refined in a sufficiently large neighborhood of a surface so that tetrahedra cut by $\Gamma(t)$ belong to the same refinement level for all $t\in[0,T]$. $\ell \in \mathbb{N}$ denotes the level of refinement and $h_\ell = \frac53\,2^{-\ell}$. The  trace $\vect P_2$--$P_1$ Taylor--Hood finite element method with BDF2 time stepping, as described in the previous section, is applied. 


\subsection{Convergence for a smooth solution}

To verify the implementation and to check the convergence order of the discrete solution, we set up an experiment with a known tangential flow along an expanding/contracting  sphere. In this example the total area of $\Gamma$ is not preserved, but it allows to  prescribe a flow $\bu$ analytically and calculate $\blf$ and $g$. 

The surface $\Gamma$ is given by its distance function 
\begin{equation}\label{dist}
	d(\bx, t) \coloneqq \| \bx \| - r(t), \quad
	r(t) \coloneqq 1 + \tfrac14 \sin(2\,\pi\,t),
\end{equation}
We consider $t \in [0,1]$.
The surface normal velocity is then $\bw_N=w_N\bn$, with
$
	w_N(t) = r'(t) =  \frac{\pi \cos(2\,\pi\,t)}{2}, $ $\bn(\bx)=\bx/|\bx|.
$
We choose~$\mu = 5 \times 10^{-3}$.

The exact solution is given by
\begin{equation}\label{exact_soln}
	\bu(\bx, t) \coloneqq \bP(\bx, t)\,(1 - 2\,t, 0, 0)^T, \quad
	p(\bx) \coloneqq  x\,y^2 + z,
\end{equation}
and right hand sides~$\blf$ and~$g = \divG \bu_T + w_N \kappa$ are computed accordingly from~\eqref{dist}--\eqref{exact_soln}. For numerical integration, exact solutions and right hand sides are extended along normal directions to $\Gamma$.

\begin{figure}[ht!]
	\centering
	\begin{subfigure}{.3\linewidth}
		\centering $t = 0$ \vskip 1mm
\includegraphicsw[0.95]{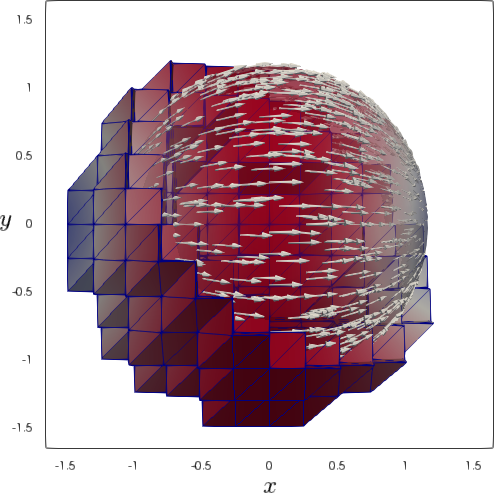}
	\end{subfigure}%
	\begin{subfigure}{.3\linewidth}
		\centering $t = 0.15$ \vskip 1mm
\includegraphicsw[.95]{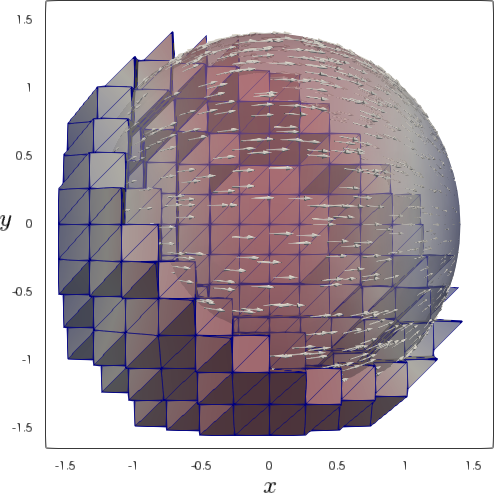}
	\end{subfigure}%
	\begin{subfigure}{.3\linewidth}
		\centering $t = 0.9$ \vskip 1mm
\includegraphicsw[.95]{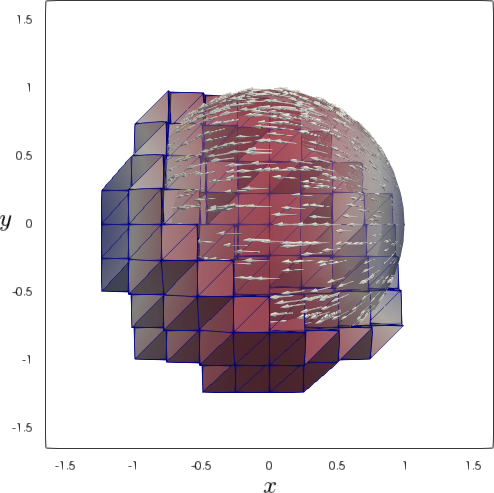}
	\end{subfigure}%
	\caption{Illustration of the extension mesh and solution at mesh level $\ell = 3$. \label{fig1}}
\end{figure}%

\begin{table}\centering
	\begin{xtabular}[1.2]{|c||c|c|c|c|}
		\hline
		Mesh level~$\ell$ & 2 & 3 & 4 & 5 \\ \hline
		$h$ & $4.17\times10^{-1}$ & $2.08\times	10^{-1}$ & $1.04\times	10^{-1}$ & $5.21\times	10^{-2}$ \\ \hline
		Averaged \# d.o.f. & $4.41\times	10^3$ & $1.73\times	10^4$ & $6.82\times	10^4$ & $2.73\times	10^5$ \\ \hline
	\end{xtabular}
	\vskip 3mm
	\begin{xtabular}[1.2]{|c|c||c|c||c||c|}
		\hline
		$\|\vect u - \vect u_h\|_{\bf 1}$	&	Order	&	$\|\vect u - \vect u_h\|_{L^2(\Gs)}$	&	Order	&	$\|p - p_h\|_{L^2(\Gs)}$	&	Order	\\ \hline
		$9.3\times	10^{-1}$	&	$\text{}$	&	$1.3\times	10^{-1}$	&	$\text{}$	&	$3.2\times	10^{-1}$	&	$\text{}$	\\ \hline
		$1.9\times	10^{-1}$	&	$2.3$	&	$9.9\times	10^{-3}$	&	$3.72$	&	$3.5\times	10^{-2}$	&	$3.2$	\\ \hline
		$4.3\times	10^{-2}$	&	$2.13$	&	$9.2\times	10^{-4}$	&	$3.42$	&	$7.3\times	10^{-3}$	&	$2.27$	\\ \hline
		$1.2\times	10^{-2}$	&	$1.92$	&	$1.2\times	10^{-4}$	&	$2.98$	&	$1.8\times	10^{-3}$	&	$2.02$	\\ \hline				
	\end{xtabular}
	\caption{Convergence results for the example with analytical solution.\label{tab1}}
\end{table}%

The numerical solution was computed on four consecutive meshes with refinement levels $\ell\in\{2,\dots,5\}$ and a time step $\Delta t=0.05$ on level 2; $\Delta t$ is halved in each spatial refinement, and for parameter in \eqref{e:delta} we set $c_\delta=1.1$. A mesh in ${\mathcal{O}}_{\Delta t}(\Gn)$ together  with the embedded $\Gamma(t)$ and computed solution are illustrated in Figure~\ref{fig1}. In Table~\ref{tab1} we show the mesh parameter $h$ and the resulting  (averaged over all time steps) number of active degrees of freedom (\# d.o.f.). We see that a mesh refinement leads to approximately four times more  degrees of freedom.
Table~\ref{tab1} further reports the velocity and pressure errors measured in (approximate) $L_{\Vd}^2$ and $L^2(\Gs)$ norms. These norms were computed using a quadrature rule for time integration. Results demonstrate the expected 2nd order convergence in the ``natural'' norms and a higher order for the velocity error in the $L^2(\Gs)$ norm. These orders are optimal for the  $\vect P_2$--$P_1$  elements used.   

\subsection{Tangential flow on a deforming sphere}
In this numerical example we consider a deforming  unit sphere  and compute the induced tangential flow, i.e., the numerical solution of the TSNSE \eqref{NSalt}.  Denote by $\Gamma_0$ the reference sphere of radius $1$  with the center  in the origin $O$. Consider spherical   coordinates $(r,\theta,\varphi) \in (0,\infty) \times [0,\pi] \times [0,2 \pi)$ and denote by $\mathcal{H}_n^m(\theta,\varphi)$, the spherical harmonic of degree $n$ and order $m$.  Assume that $\mathcal{H}_n^m$ are normalized, i.e. $\|\mathcal{H}_n^m\|_{L^2(\Gamma_0)}=1$.  For the evolving surface we consider as ansatz
\begin{equation}\label{osc}
    \Gamma(t)=\left\{ \,\bx=(r,\theta,\varphi)~|~r= 1+\sum_{n=1}^N \sum_{|m|\le n} A_{n,m}(t)\mathcal{H}_n^m(\theta,\varphi)\, \right\},
\end{equation}
with suitably chosen coefficients $A_n(t)$.  The function $\xi:=\sum_{n=1}^N \sum_{|m|\le n} A_{n,m}(t)\mathcal{H}_n^m(\theta,\varphi)$ describes the radial deformation. We assume small oscillations,   $\|\xi\|\ll 1$.  
Under this assumption, an accurate \emph{approximation} of the normal velocity is given by  $\bw_N=w_N\bn$, with
\begin{equation} \label{Errw}
w_N= \frac{d \xi}{dt}=\sum_{n=1}^N \sum_{|m|\le n}\frac{d A_{n,m}}{dt}\mathcal{H}_n^m, \quad \bn(\bx)=\bx/|\bx|.
\end{equation}
We want the surface to be inextensible, i.e., $\frac{d}{dt} |\Gamma(t)|=0$. Appropriate coefficients $A_{n,m}(t)$ such that we have inextensibility can be determined as follows.
Application of  the surface Reynolds transport formula and integration by parts gives for the variation of surface area:
\begin{equation}\label{varSurf}
	\frac{d}{dt} |\Gamma(t)|=\frac{d}{dt}  \int_{\Gamma(t)} 1 \,ds=  \int_{\Gamma(t)} \mbox{div}_{\Gamma} \mathbf{w}_N\,ds=
	\int_{\Gamma(t)} \kappa w_N \,ds.
\end{equation}
For the doubled mean curvature we have, cf.~\cite{lamb1924hydrodynamics}, 
\[
\begin{aligned}
	\kappa&={2}-{2\xi}-\Delta_{\Gamma}\xi
	={2}-\sum_{n=1}^N\sum_{|m|\le n} \left\{{2A_{n,m}\mathcal{H}_n^m}-{n(n+1)} A_{n,m}\mathcal{H}_n^m\right\}\\
	&={2}+\sum_{n=1}^N\sum_{|m|\le n}{A_{n,m}}(n(n+1)-2)\mathcal{H}_n^m.
\end{aligned}
\]
Using $\int_{\Gamma_0}\mathcal{H}_n^m=0$, $n\ge1$, and  $\int_{\Gamma_0}\mathcal{H}_n^m\mathcal{H}_{n'}^{m'}=\delta_n^{n'}\delta_m^{m'}$,  we compute for the area variation:
\begin{equation}\label{RHSvar}
	\frac{d}{dt} |\Gamma(t)|=\int_{\Gamma(t)} \kappa w_N ds=\sum_{n=1}^N\sum_{|m|\le n}{(n-1)(n+2)}\frac{dA_{n,m}}{dt}A_{n,m}.
\end{equation}
Based on this formula we set $A_{2,0}= \frac\epsilon2\cos(\omega t)$, $A_{3,0}= \frac\epsilon{\sqrt{10}}\sin(\omega t)$, and $A_{n,m}=0$ for other coefficients. For this choice of coefficients one easily verifies $\frac{d}{dt} |\Gamma(t)|=0$.
The TSNSE equations \eqref{NSalt} are then solved with the right hand side  given by \eqref{rhs} with $u_N=w_N$ computed from~\eqref{Errw}. The initial velocity is zero.

In the first numerical example we let  $\epsilon = 0.2$,  $\omega = 2\pi$, $\mu = \frac12 10^{-4}$,  and include  $\mathcal{H}_2^0$ and $\mathcal{H}_3^0$, two zonal spherical harmonics of degree 2 and 3. The relative variation of the surface area $\Gamma(t)$ in the left plot in Figure~\ref{fig:area} shows less than $0.1\%$ of surface variation, which is due to approximation errors and  finite (rather than infinitesimal) deformations. The latter causes an approximation error in \eqref{Errw}. 
\begin{figure}[ht!]
	\centering\includegraphicsw[0.45]{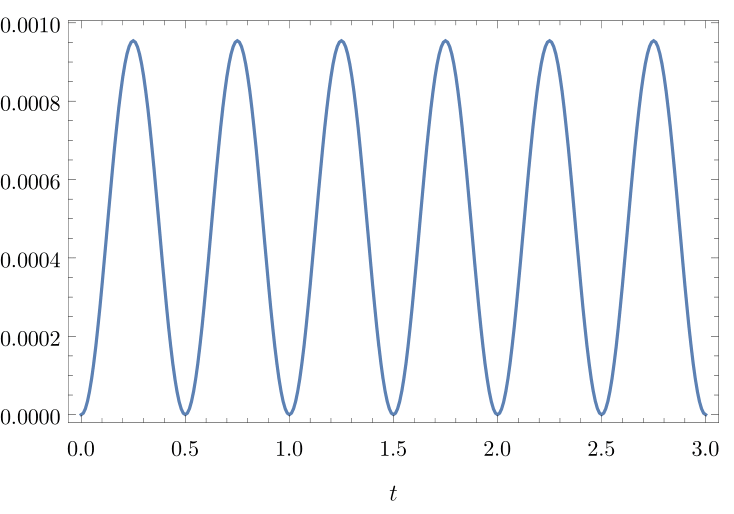}\quad 
	\includegraphicsw[0.45]{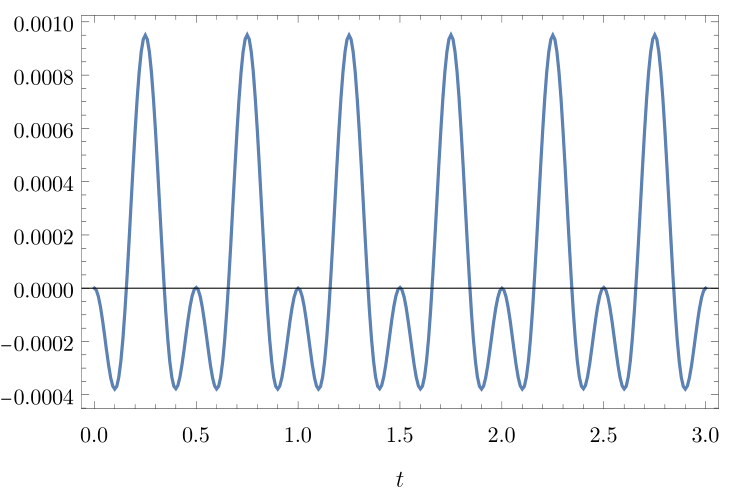}
	\caption{Relative surface area variation $\frac{|\Gamma(0)|-|\Gamma(t)|}{|\Gamma(0)|}$ as a function of time for axisymmetric (left plot) and asymmetric (right plot) deformations of the sphere. \label{fig:area}}
\end{figure}

The velocity field induced by these axisymmetric  deformations of the sphere  is visualized in Figure~\ref{fig:vel1}. We see that the velocity pattern is dominated by a sink-and-source flow driven by the term $-\kappa w_N$ on the right-hand side of the divergence condition in  \eqref{NSalt}. 

\begin{figure}[ht!]
	\centering
	\begin{subfigure}{.3\linewidth}
		\centering $t = 0$ \vskip 1mm
		\includegraphicsw{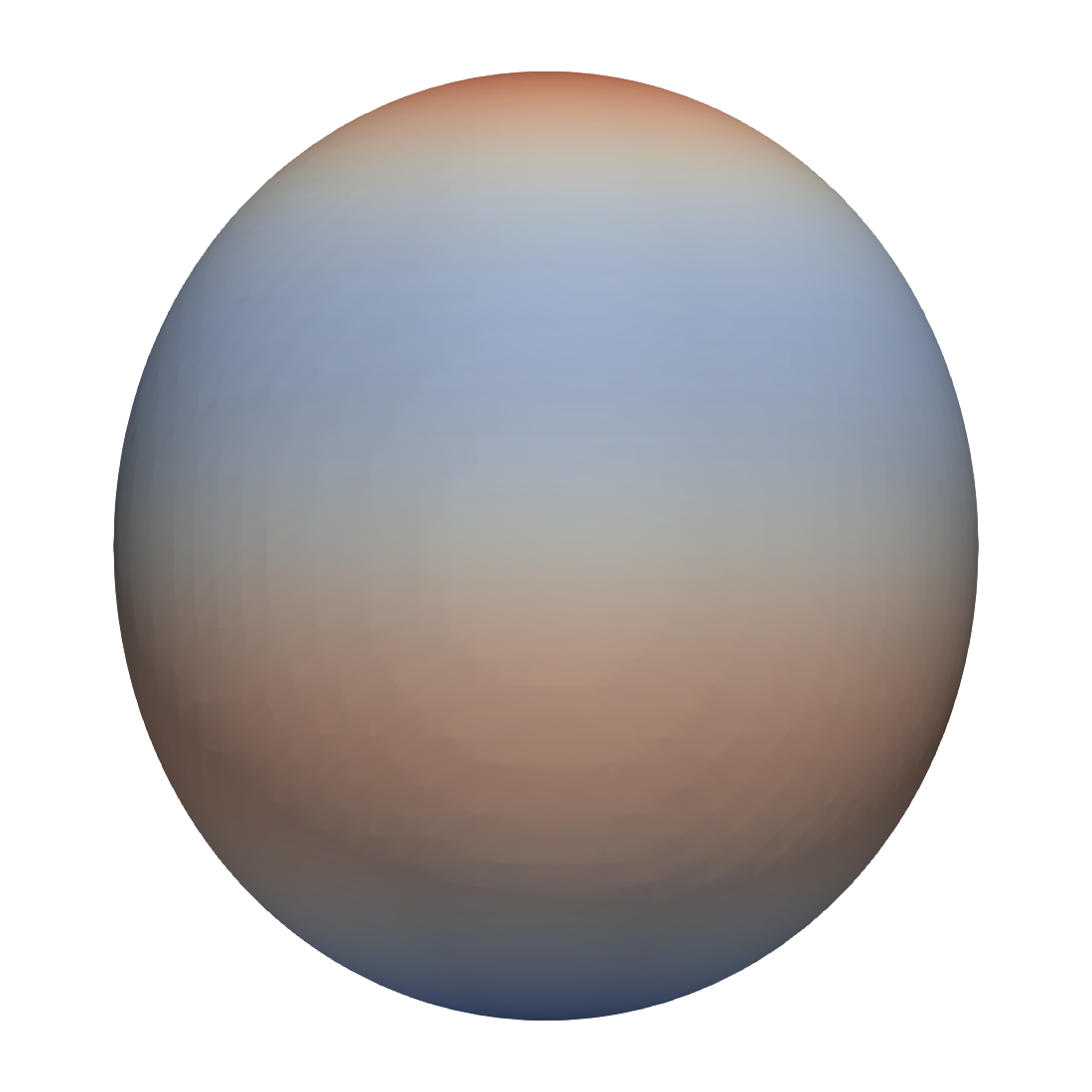}
	\end{subfigure}%
	\begin{subfigure}{.3\linewidth}
		\centering $t = 0.1$ \vskip 1mm
		\includegraphicsw{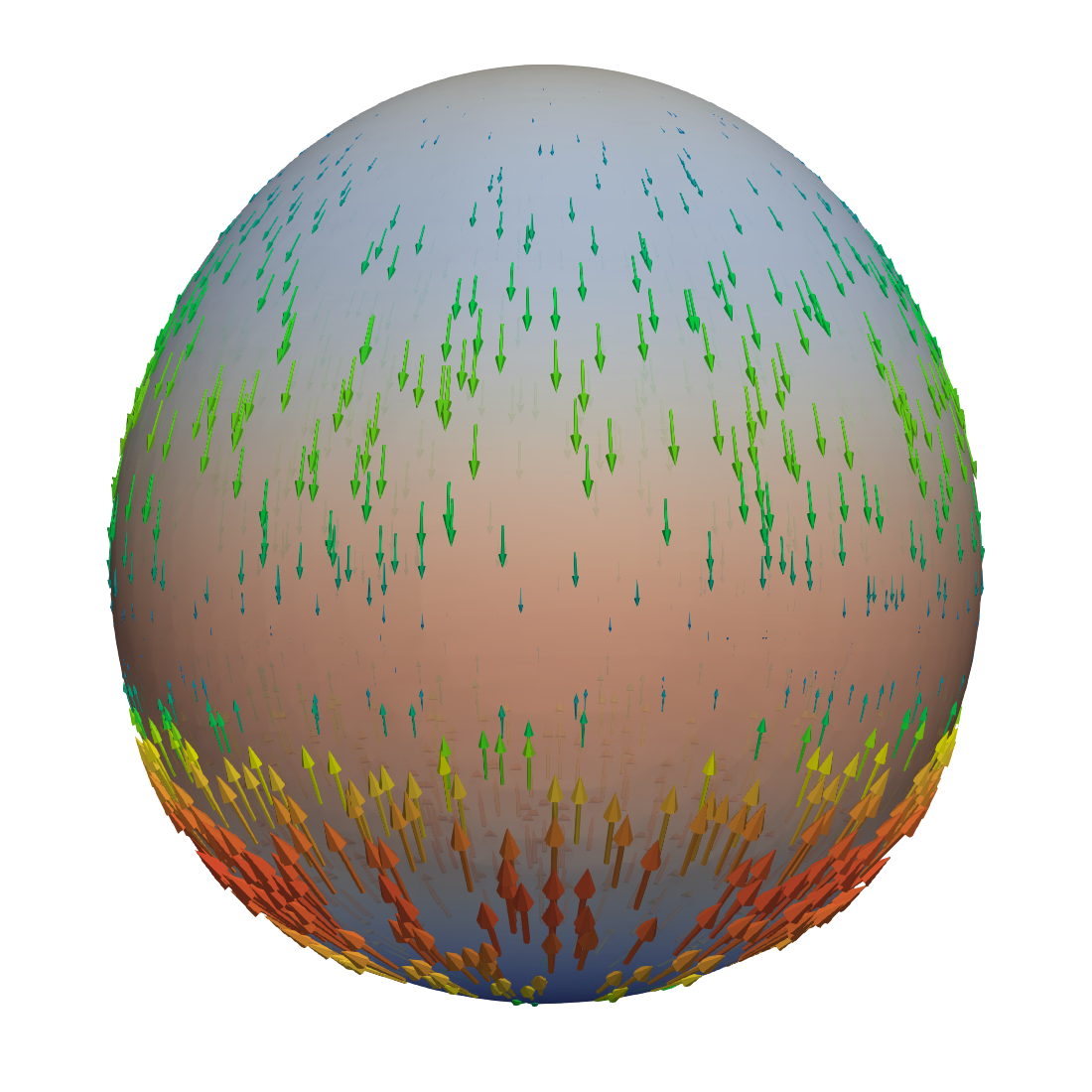}
	\end{subfigure}%
	\begin{subfigure}{.3\linewidth}
		\centering $t = 0.3$ \vskip 1mm
		\includegraphicsw{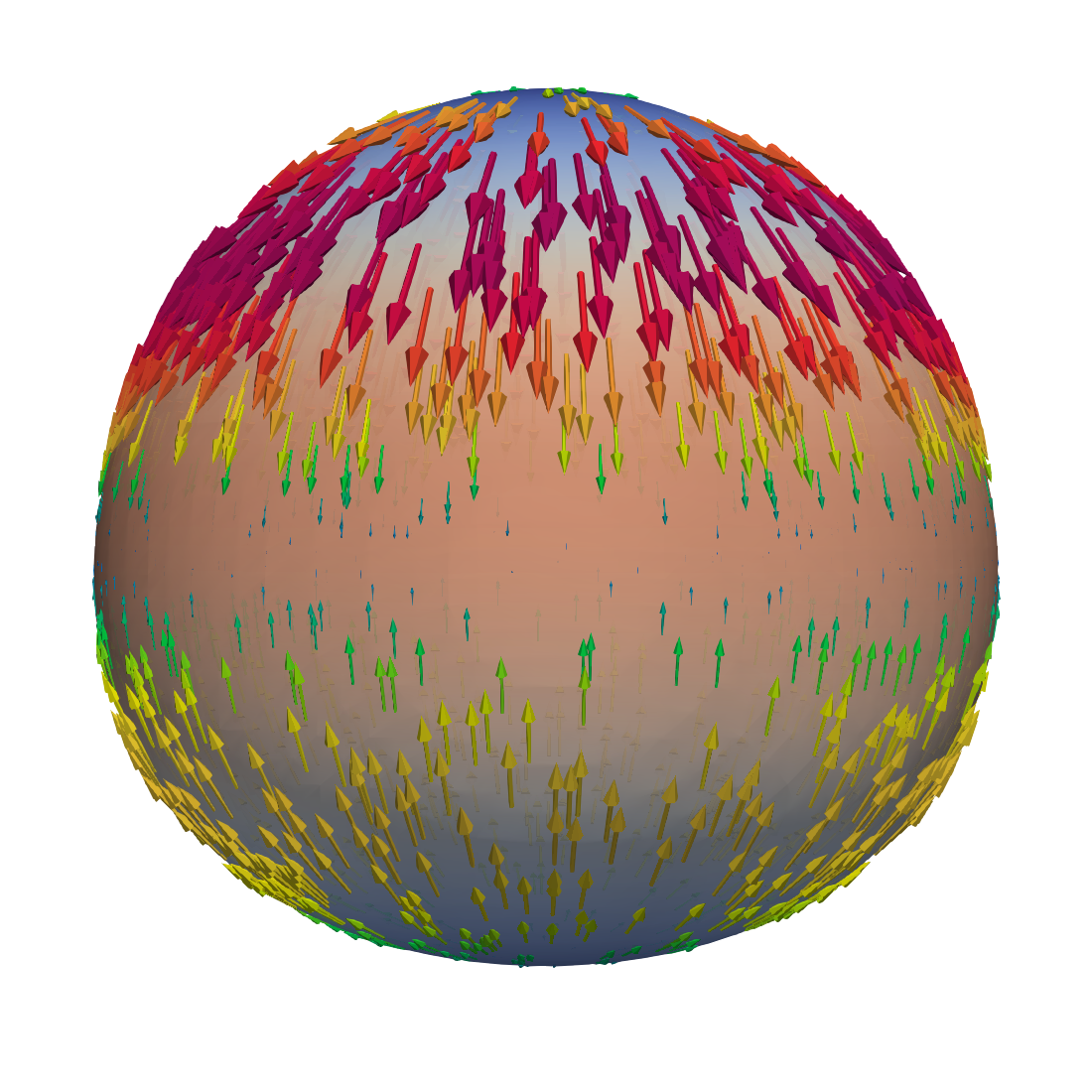}
	\end{subfigure}%
	\vskip 2mm
	\begin{subfigure}{.3\linewidth}
		\centering $t = 0.5$ \vskip 1mm
		\includegraphicsw{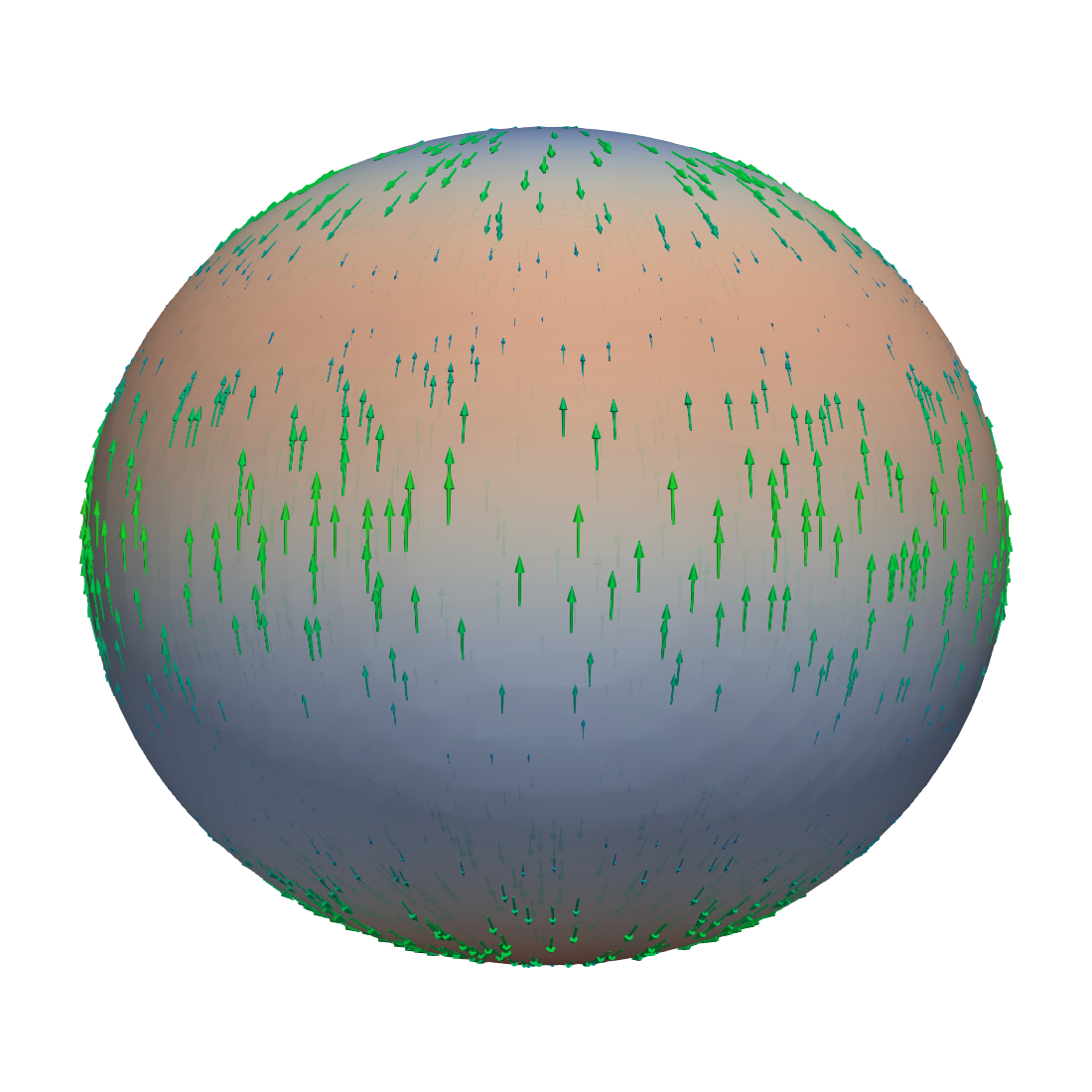}
	\end{subfigure}%
	\begin{subfigure}{.3\linewidth}
		\centering $t = 0.7$ \vskip 1mm
		\includegraphicsw{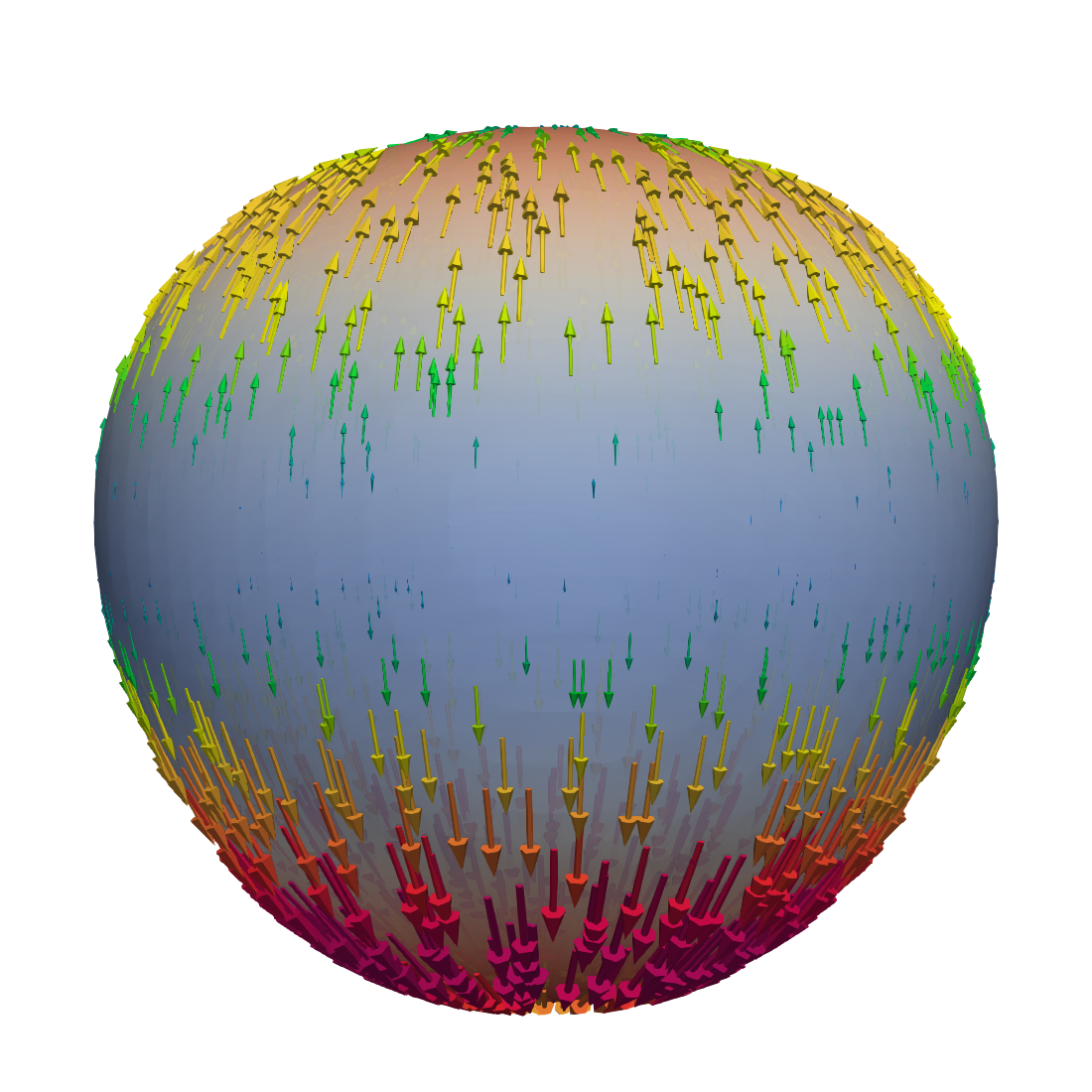}
	\end{subfigure}%
	\begin{subfigure}{.3\linewidth}
		\centering $t = 0.9$ \vskip 1mm
		\includegraphicsw{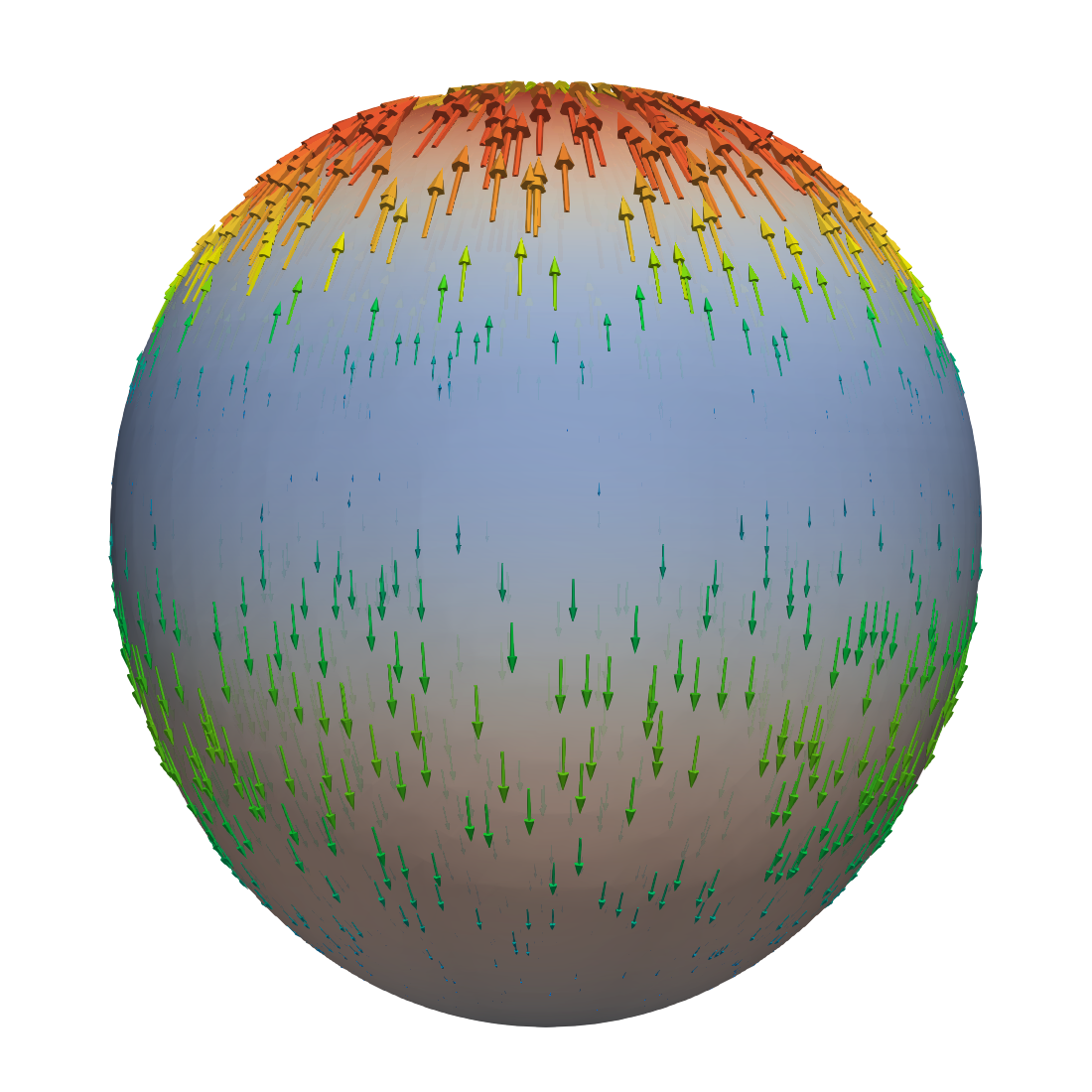}
	\end{subfigure}%
	\vskip 2mm
	\begin{subfigure}{.3\linewidth}
		\centering $t = 1$ \vskip 1mm
		\includegraphicsw{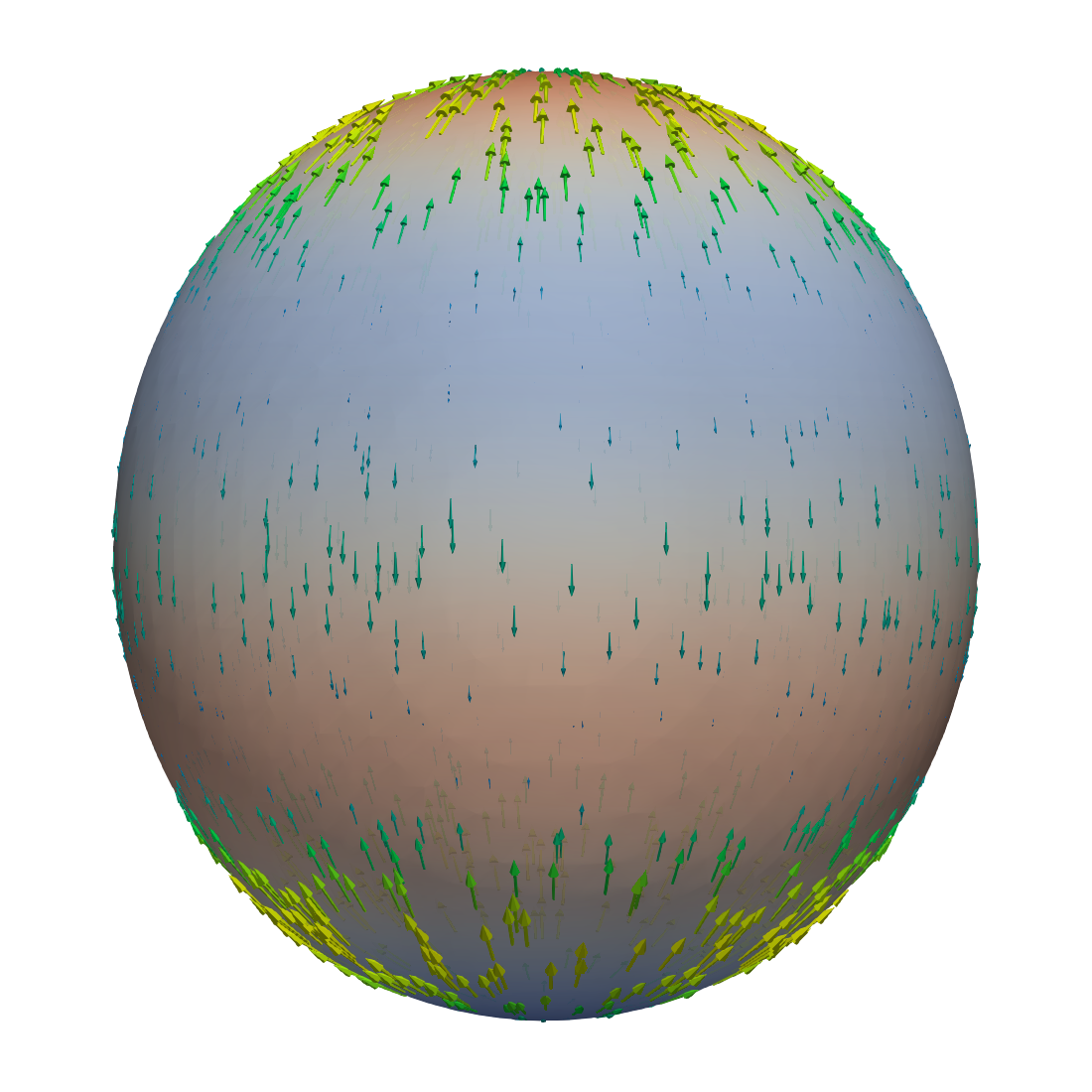}
	\end{subfigure}%
	\begin{subfigure}{.6\linewidth}
		\centering\includegraphicsw[.9]{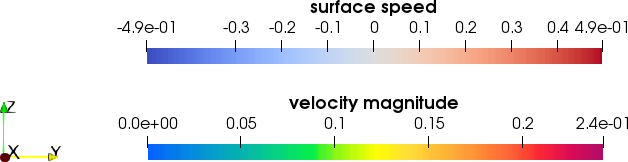}
	\end{subfigure}%
	\caption{Visualization of velocity field for axisymmetric deformations of the sphere ; mesh level $\ell = 4$, $\Delta t = 0.01$. Click \href{https://youtu.be/v6naA_6Urck}{here} to see the full animation. \label{fig:vel1}}
\end{figure}%

%
%

\begin{figure}[ht!]
	\centering
	\begin{subfigure}{.3\linewidth}
		\centering $t = 0$ \vskip 1mm
		\includegraphicsw{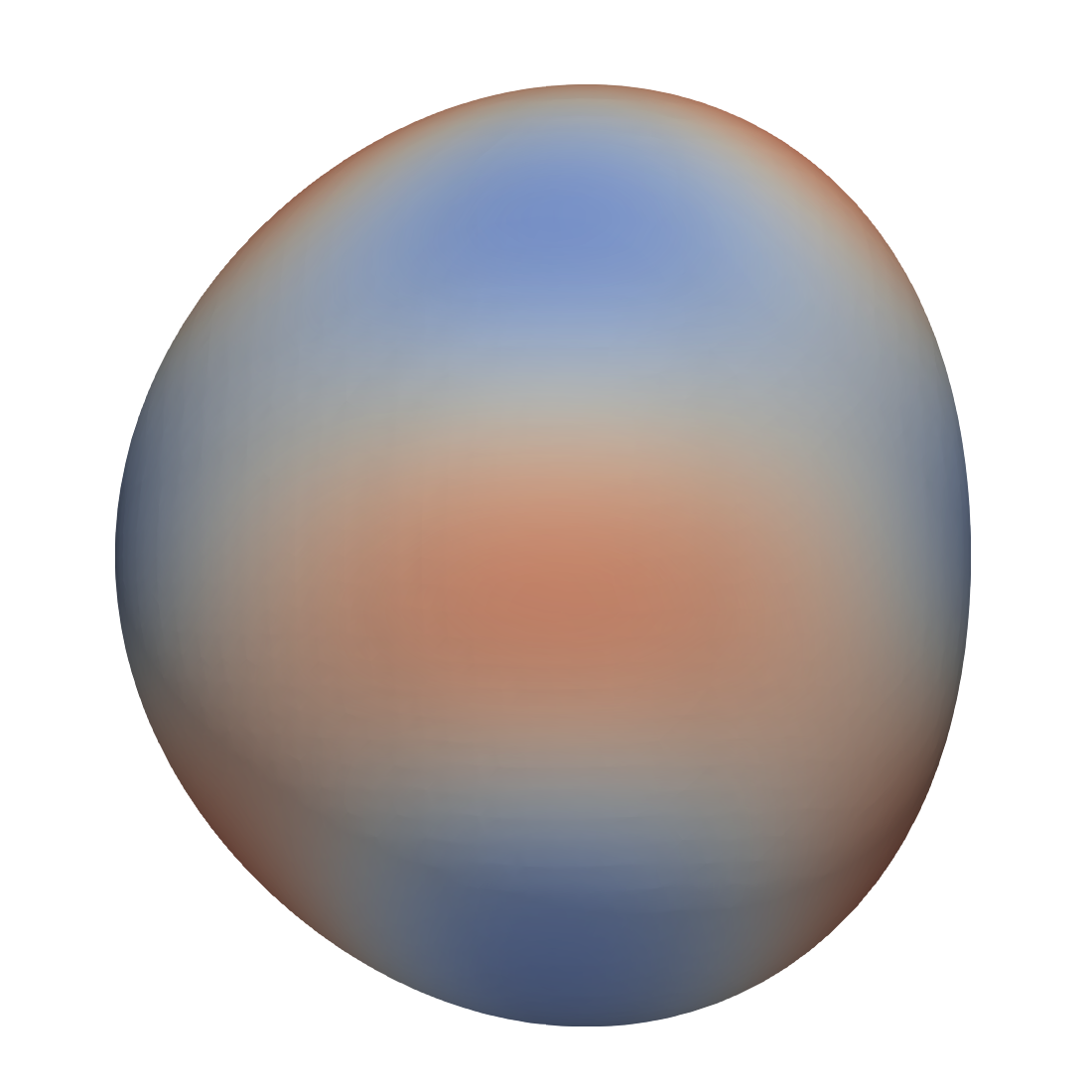}
	\end{subfigure}%
	\begin{subfigure}{.3\linewidth}
		\centering $t = 0.1$ \vskip 1mm
		\includegraphicsw{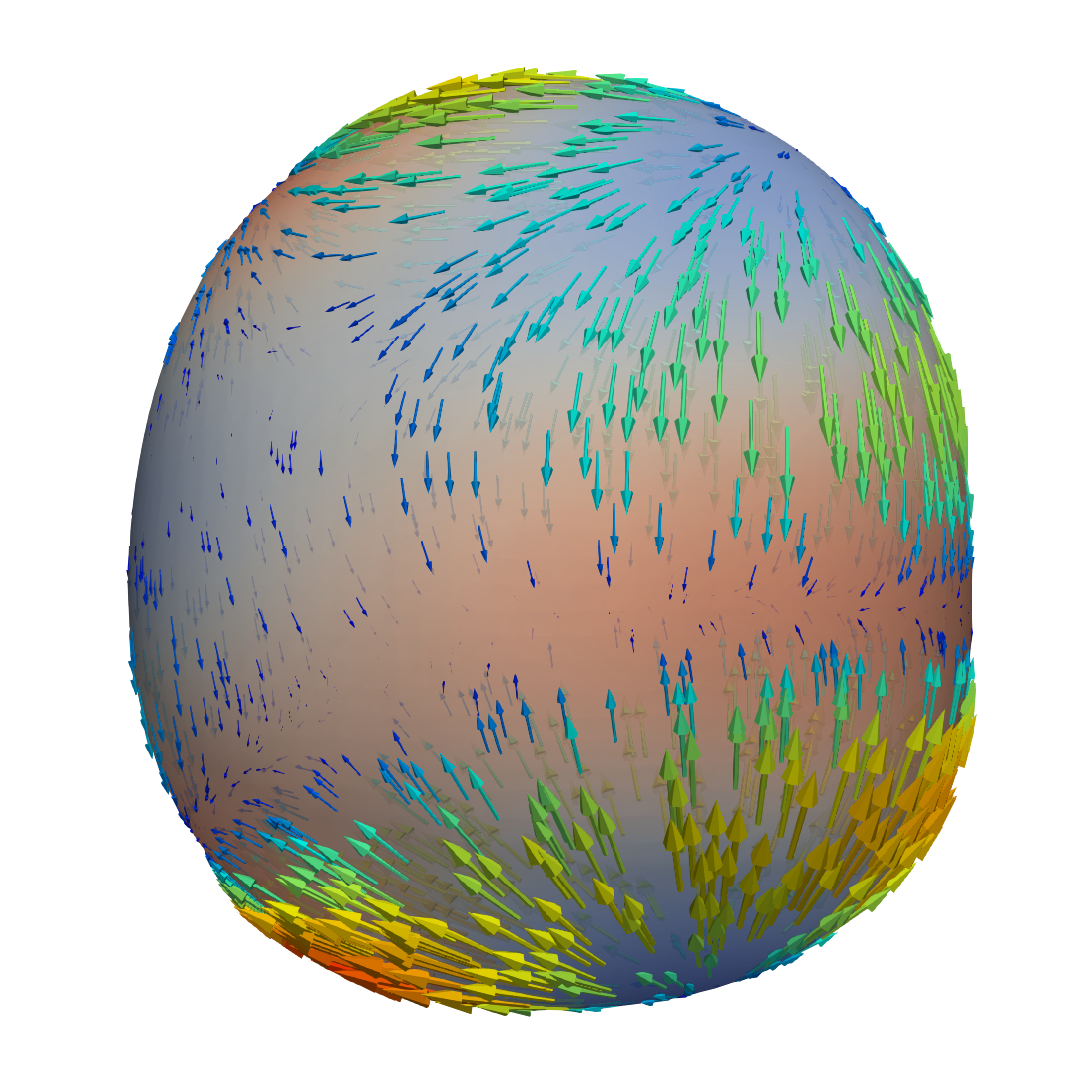}
	\end{subfigure}%
	\begin{subfigure}{.3\linewidth}
		\centering $t = 0.3$ \vskip 1mm
		\includegraphicsw{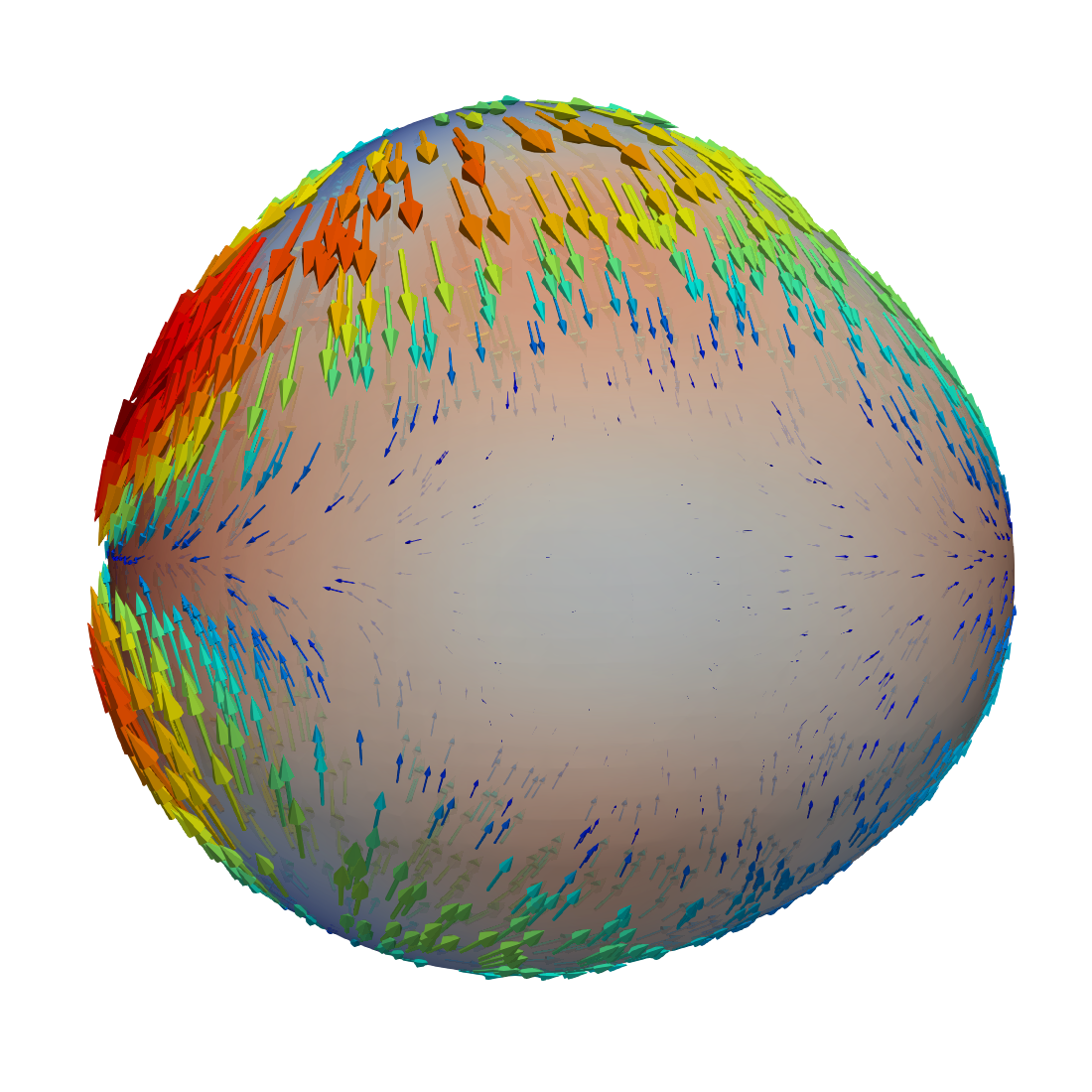}
	\end{subfigure}%
	\vskip 2mm
	\begin{subfigure}{.3\linewidth}
		\centering $t = 0.5$ \vskip 1mm
		\includegraphicsw{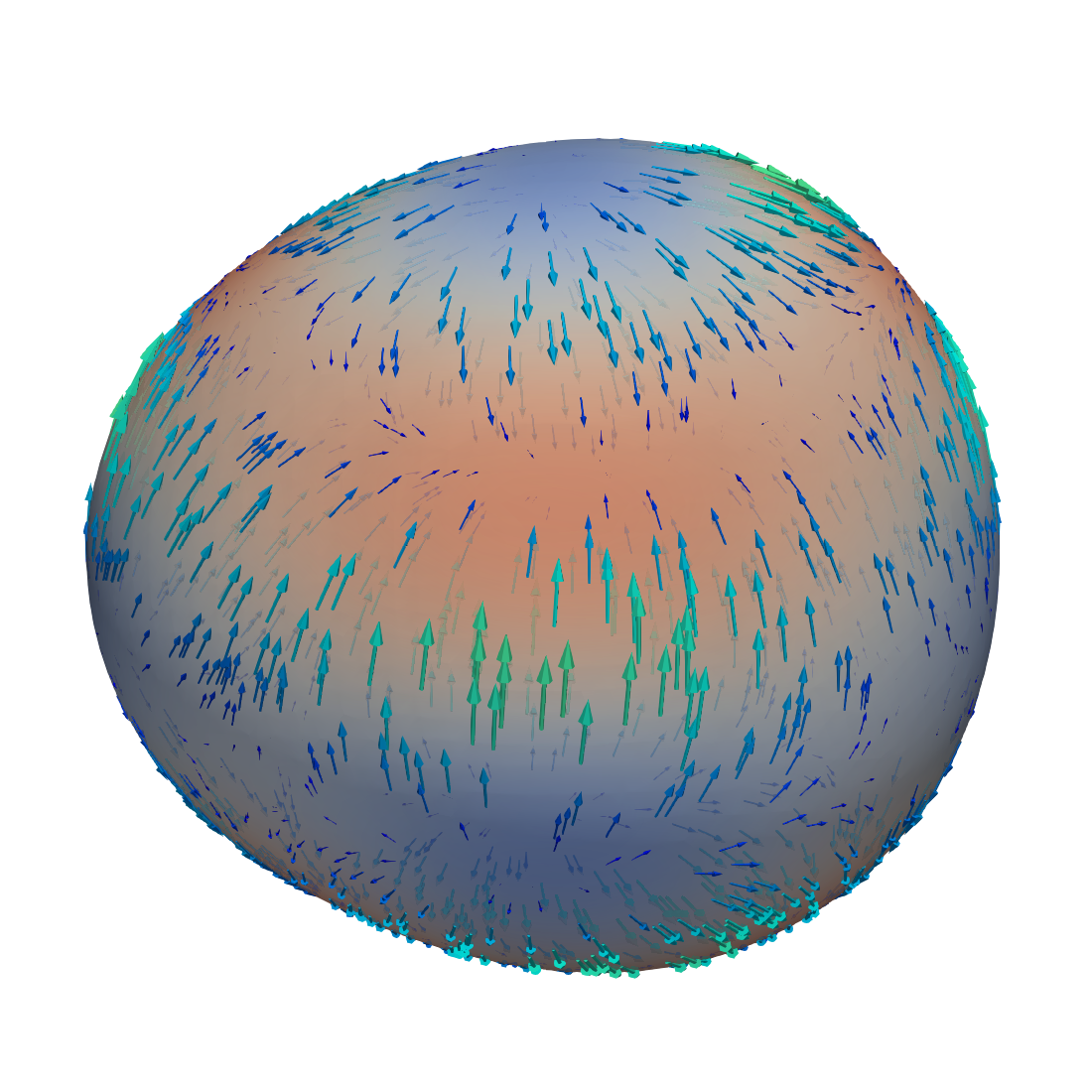}
	\end{subfigure}%
	\begin{subfigure}{.3\linewidth}
		\centering $t = 0.7$ \vskip 1mm
		\includegraphicsw{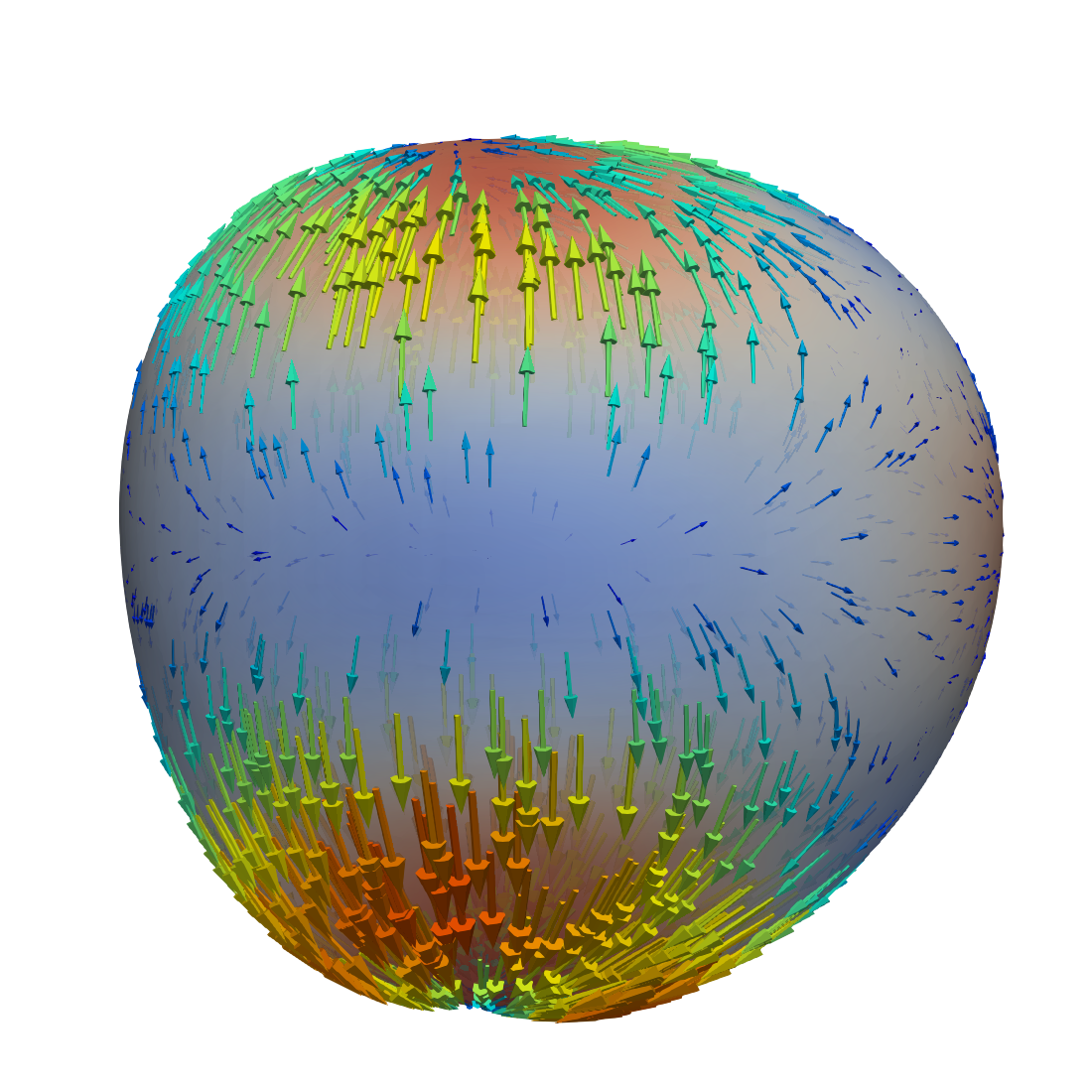}
	\end{subfigure}%
	\begin{subfigure}{.3\linewidth}
		\centering $t = 0.9$ \vskip 1mm
		\includegraphicsw{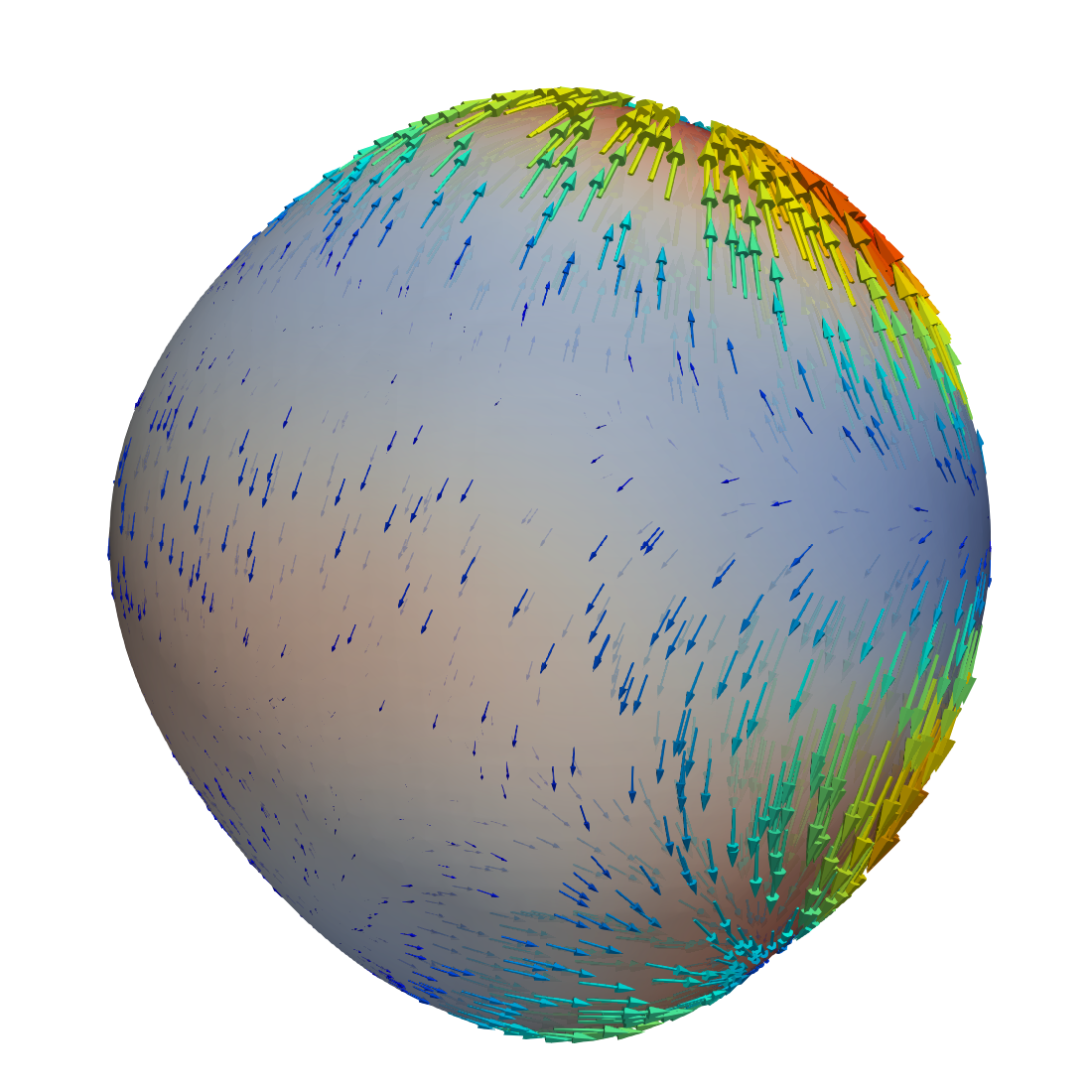}
	\end{subfigure}%
	\vskip 2mm
	\begin{subfigure}{.3\linewidth}
		\centering $t = 1$ \vskip 1mm
		\includegraphicsw{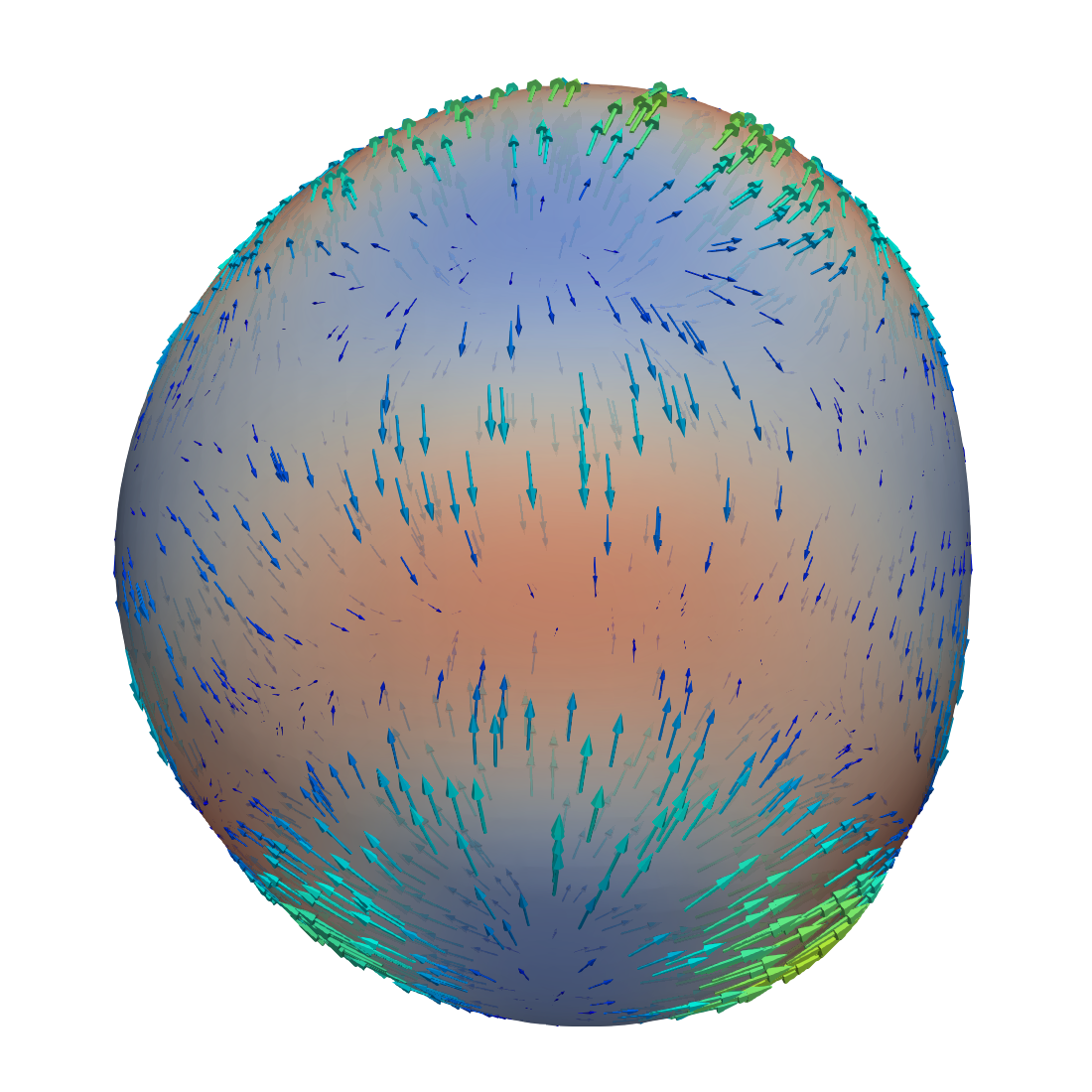}
	\end{subfigure}%
	\begin{subfigure}{.6\linewidth}
		\centering\includegraphicsw[.9]{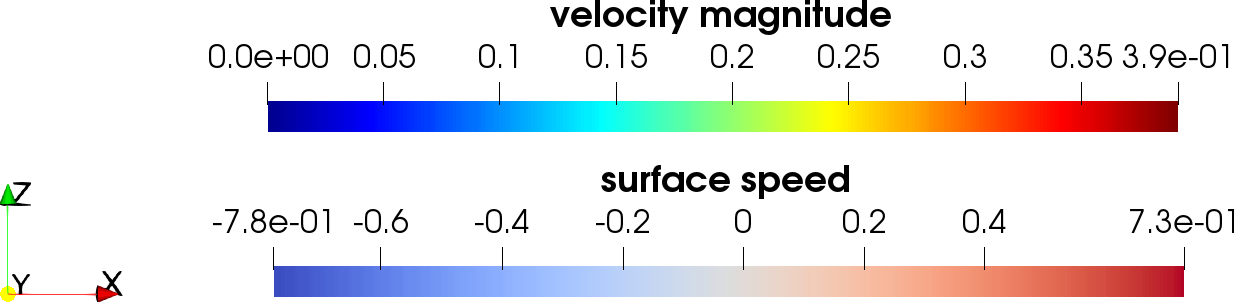}
	\end{subfigure}%
	\caption{Visualization of velocity field for asymmetric deformations of the sphere ; mesh level $\ell = 4$, $\Delta t = 0.01$. Click \href{https://youtu.be/89Dky1aaE_M}{here} to see the full animation. \label{fig:vel2}}
\end{figure}%

We repeat the experiment, but decrease the viscosity to $\mu=\frac12 10^{-5}$ and add two more spherical
harmonics, the sectorial  $\mathcal{H}_{3}^1$ harmonic and the tesseral $\mathcal{H}_{4}^{2}$ one, to make the deformation   \emph{a}symmetric. The radial displacement in this experiment is then given by
\[ \begin{split}
	\xi & = 0.2 \left( \tfrac12\cos(2\pi\,t)\,\mathcal{H}_{2}^{0}(\bx) + \tfrac1{\sqrt{10}}\sin(2\pi\,t)\,\mathcal{H}_{3}^{0}(\bx) \right), \\ 
	& + 0.1 \left( \tfrac12\cos(4\pi\,t)\,\mathcal{H}_{3}^{1}(\bx) + \tfrac5{18}\sin(4\pi\,t)\,\,\mathcal{H}_{4}^{2}(\bx) \right).
\end{split} \]
Again, the coefficients are such that $\frac{d}{dt} |\Gamma(t)|=0$ according to equation \eqref{RHSvar}.
The resulting velocity field  is visualized in Figure~\ref{fig:vel2}. The velocity pattern is still dominated by the sink-and-source flow. Note that in both cases there are no outer forces and the flow is completely ``geometry driven''.

\bibliographystyle{siam}
\bibliography{literatur}{}


\end{document}